\documentclass[11pt]{article}
\usepackage[margin=1.3in]{geometry}
\usepackage{graphicx} % Required for inserting images
\usepackage{authblk}  % For better author formatting
\usepackage{hyperref} % For clickable email links
\usepackage[dvipsnames]{xcolor}
\usepackage{amsmath, amssymb, amsfonts,amsthm}

\usepackage{subcaption}
\usepackage{graphicx}
\usepackage[dvipsnames]{xcolor}
\usepackage{url}
\usepackage{caption}
\usepackage{bm}
\usepackage{array} % Required for tabular

\usepackage{nicefrac}
\usepackage{graphicx}      % include this line if your document contains figures
\usepackage{tikz}
\usepackage{pgfplots}
\pgfplotsset{compat=1.8}
\usepackage{pgfplotstable}
\usetikzlibrary{3d, calc, decorations.markings}
\usepackage{placeins}    % Ensures figures do not float past barriers
\pgfplotsset{compat=1.17}

\usepackage[shortlabels]{enumitem}

\newcommand{\N}{\mathbb{N}}
\newcommand{\C}{\mathbb{C}}

\newcommand{\diag}{\text{diag}}
\DeclareMathOperator{\sgn}{sgn}

\newtheorem{theorem}{Theorem}[section]
\newtheorem{proposition}[theorem]{Proposition}
\newtheorem{lemma}[theorem]{Lemma}
\newtheorem{corollary}[theorem]{Corollary}

\theoremstyle{definition}
\newtheorem{definition}{Definition}[section]
\newtheorem{assumption}{Assumption}[section]

\theoremstyle{remark}
\newtheorem{remark}{Remark}[section]

\usepackage[style= numeric-comp,hyperref=true, doi=false,url=false,
            isbn=false,
            firstinits=true, sorting = none, 
            block=none, backend=bibtex,maxnames=99]{biblatex}
\renewbibmacro{in:}{} 
\bibliography{references}

\makeatletter
\def\blfootnote{\xdef\@thefnmark{}\@footnotetext}
\makeatother

\makeatletter
\newcommand{\oset}[3][0ex]{%
  \mathrel{\mathop{#3}\limits^{
    \vbox to#1{\kern-1\ex@
    \hbox{$\scriptstyle#2$}\vss}}}}
\makeatother

\title{Feedback Diagonalization and Stabilization for Discrete Linear Ensemble Systems}

\begin{document}
\author{Jonathan M. Bosnich \quad \mbox{and} \quad Xudong Chen
}
\date{}
\maketitle
\blfootnote{J. M. Bosnich and X. Chen are with the Department of Electrical and Systems Engineering, Washington University in St. Louis, St. Louis, MO. Emails: \texttt{\{bosnich,cxudong\}@wustl.edu}.}
\blfootnote{Corresponding author: J. M. Bosnich.}

%%%%%%%%%%%%%%%%%%%%%%%%%% Option 1 %%%%%%%%%%%%%%%%%%%%%%%%%%%
\begin{abstract}
We consider a countably infinite collection of linear, scalar control systems, where each system is represented by the pair $(a_n,b_n)$ with $a_n>0$ for $n\in \N$, and all systems are forced by a common scalar control input. We refer to this collection of systems as a \textit{discrete linear ensemble system}. The ensemble system is \textit{feedback stabilizable} if there exists a common feedback control input that asymptotically stabilizes every system simultaneously. Unlike finite-dimensional linear systems, an infinite-dimensional linear system is not guaranteed to be stable if its poles lie in the open left half-plane. Stability is guaranteed, however, if \textit{(i)} its poles are contained in the closed left half-plane and \textit{(ii)} its infinitesimal generator is diagonalizable. In this paper, we provide necessary and sufficient conditions for the existence of a static, linear feedback control law that ensures the closed-loop ensemble system satisfies \textit{(i)} and \textit{(ii)}. In particular, we show that the exponential decay of the sequences $(|b_n|)_{n\in \N}$ and $(a_n/|b_n|)_{n\in \N}$ is necessary and, under an assumption on the desired poles, sufficient. 
\end{abstract}

%%==================================%%
%% Introduction
%%==================================%%

\section{Introduction}\label{sec: intro}
Ensemble control theory is the study of steering a large (in the limit, infinite) population of individual systems using a single finite-dimensional control input applied uniformly across the entire population. Since ensemble control influences networks of dynamical systems at the population level, it is a framework well suited for controlling large, complex systems, e.g., quantum spin systems controlled by a magnetic field \cite{glaser1998unitary, brockett2000stochastic, li2006control}, neural systems controlled by optogenetic stimulation \cite{ching2013control, mardinly2018precise}, molecular systems controlled by the stimulus of light \cite{yanlei2003photomechanics} and heat \cite{taniguchi2018walking}, and swarms of heterogeneous robots controlled by a single broadcast signal \cite{becker2014controlling, becker2017controlling}. Motivated by its broad applicability, there has been active development of the mathematical foundations of ensemble control theory. Fundamental control concepts have been studied for a variety of ensemble systems, including controllability \cite{triggiani1975stabilizability, chen2023controllability, agrachev2016ensemble}, observability \cite[Chapter~12]{fuhrmann2015mathematics}, \cite{chen2019structure, chen2020ensemble}, and optimal control \cite{scagliotti2023optimal}, often extending classical results from the finite-dimensional case. However, a precise understanding of when an ensemble system is \textit{feedback stabilizable} remains a challenging open problem, largely because finite-dimensional results do not extend to the infinite-dimensional setting. We will expound upon this point shortly. Recently, a partial solution to this open problem was provided by Chen in \cite{chen2025poleplacement}. In this paper, we extend and generalize the results of this prior work. We begin by introducing the ensemble system addressed in \cite{chen2025poleplacement}.

\subsection{Problem formulation}\label{subsec: setup}
Consider the following countably infinite ensemble of linear, scalar systems with a common scalar control input:
\begin{align}
    \dot x_n(t) = a_nx_n(t) +b_n u(t), \quad \text{for all } n\in \N,
    \label{eqn: the individual systems}
\end{align}
where $\N$ denotes the set of positive integers. We define the sequence\footnote{We use $x = (x_n)$ or simply $(x_n)$ to denote the infinite sequence $(x_n)_{n\in \N}$.} $x(t) := (x_n(t)) \in X$, where $X$ can be any of the following Banach sequence spaces over $\C$: $\ell^p$ for $1\leq p \leq \infty$; $c$, the space of convergent sequences; or $c_0$, the space of null sequences, i.e., sequences that converge to zero. We denote by $\|\cdot\|_{\ell^p}$ the $\ell^p$-norm for $1\leq p \leq \infty$, and we equip $c$ and $c_0$ with the $\ell^\infty$-norm, i.e., the sup norm. One can think of the sequence $x(t)$ as the state of the ensemble system with $X$ the state space. The ensemble system is defined by the sequences $a := (a_n) \in \ell^\infty$ and $b := (b_n) \in X$. For the case where $X=c$, we further require $a\in c$ to ensure that $x(t)\in X$ for all $t\geq 0$. Finally, $u(t)$ is a complex-valued and locally integrable scalar control input applied uniformly to each individual system.  We can rewrite \eqref{eqn: the individual systems} as the following infinite-dimensional linear control system:
% %% more concise version for SIAM below %%
% \noindent We define the sequence\footnote{We use $x = (x_n)$ or simply $(x_n)$ to denote the infinite sequence $(x_n)_{n\in \N}$.} $x(t) := (x_n(t)) \in X$, where $X$ can be any of the following Banach sequence spaces over $\C$: $\ell^p$ for $1\leq p \leq \infty$; $c$, the space of convergent sequences; or $c_0$, the space of null sequences, i.e., sequences that converge to zero. We denote by $\|\cdot\|_{\ell^p}$ the $\ell^p$-norm for $1\leq p \leq \infty$, and we equip $c$ and $c_0$ with the $\ell^\infty$-norm, i.e., the sup norm. One can think of the sequence $x(t)$ as the state of the ensemble system with $X$ the state space. Let $a:=(a_n)\in \ell^\infty$, $b:=(b_n)\in X$, and $u(t)$ be complex-valued and locally integrable. For the case where $X=c$, we further require $a\in c$ to ensure that $x(t)\in X$ for all $t\geq 0$. We can rewrite \eqref{eqn: the individual systems} as the following infinite-dimensional linear control system:
 \begin{align}
     \dot x(t) = Ax(t) + bu(t),
     \label{eqn: infinite-dim ensemble system}
 \end{align}
  where $A : X \to X$ is a diagonal operator with rule $(x_n) \mapsto (a_nx_n)$, which we represent as the infinite-dimensional diagonal matrix $A=\diag(a_1, a_2, \ldots)$. Let $\mathcal{B}(X)$ be the space of all bounded linear operators from $X$ to $X$ and denote by $\|\cdot \|_{\mathcal{B}(X)}$ the operator norm. Note that $A\in \mathcal{B}(X)$ since $(a_n)\in \ell^\infty$. Let $X^*$ be the dual space of $X$, i.e., the space of all bounded linear functionals from $X$ to $\C$. Applying linear feedback control $u(t) = kx(t)$ for some $k \in X^*$, system \eqref{eqn: infinite-dim ensemble system} becomes
 \begin{align}
    \dot x(t) = T_kx(t), \quad \text{where } T_k :=A+bk.
    \label{eqn: closed-loop ensemble system}
\end{align}
We refer to \eqref{eqn: closed-loop ensemble system} as the \textit{closed-loop} system and $T_k$ as its infinitesimal generator. Note that $T_k$ is a rank-one perturbation of $A$, and hence $T_k \in \mathcal{B}(X)$. 

We assume that the ensemble system \eqref{eqn: the individual systems} is unstable in the absence of control input, i.e., when $u(t)=0$. While this assumption only requires that at least one individual system is unstable, for our purposes, we shall assume that all individual systems are unstable.
\begin{assumption}
    Each $a_n$ is a positive real number.
    \label{ass: positive a_n}
\end{assumption} 

\noindent We will establish necessary and sufficient conditions for the existence of a $k\in X^*$ such that system \eqref{eqn: closed-loop ensemble system} is stable, where we use the following notion of stability.
\begin{definition}
    System \eqref{eqn: closed-loop ensemble system} is \textbf{stable} if there exists a constant $C > 0$ such that for every $x(0) \in X$ and every $t\geq 0$, it holds that $\| x(t) \|_X \leq C\|x(0)\|_X$. Further, system \eqref{eqn: closed-loop ensemble system} is \textbf{asymptotically stable} if it is stable and, additionally, $\lim_{t\to \infty} x(t) = 0$ for any initial condition $x(0) \in X$. 
\end{definition}

\begin{remark}
     We mention a third, stronger notion of stability: system~\eqref{eqn: closed-loop ensemble system} is \textit{exponentially stable} if there exists constants $M\geq 1$ and $\alpha>0$ such that for every $x(0)\in X$ and every $t\geq 0$, it holds that $\|x(t)\|_X \leq Me^{-\alpha t} \|x(0)\|_X$. We note that under Assumption~\ref{ass: positive a_n}, the closed-loop system~\eqref{eqn: closed-loop ensemble system} can never be exponentially stable (see, e.g.,~\cite[Chapter 5]{curtain2012introduction}).
\end{remark}

To understand why the stability of system \eqref{eqn: closed-loop ensemble system} differs starkly from that of a finite-dimensional linear system, let us recall well-known necessary and sufficient conditions for the (asymptotic) stability of linear systems. Let $\Sigma(T_k)$ denote the spectrum of $T_k$, which we will also refer to as the set of \textit{poles} of $T_k$. Then, a necessary condition for the stability of system \eqref{eqn: closed-loop ensemble system} is that  
\begin{align}
    \Sigma (T_k) \subseteq H:=\{z \in \C \mid \text{Re}(z)\leq 0 \}.
    \label{cond: stable poles}
\end{align}
For our purposes, we refer to the existence of a $k\in X^*$ such that \eqref{cond: stable poles} is satisfied as \textit{pole placement} (in the closed left half-plane). 
Next, we recall a sufficient condition for asymptotic stability of \textit{finite-dimensional} linear systems. Let $\Sigma_{\text{disc}}(T_k) \subseteq \Sigma(T_k)$ be the discrete spectrum of $T_k$. If $T_k$ is the infinitesimal generator of a finite-dimensional linear system, then $\Sigma(T_k) = \Sigma_{\text{disc}}(T_k)$ and asymptotic stability of the system is guaranteed if it holds that
\begin{align}
    \Sigma_{\text{disc}} (T_k) \subseteq H^- :=\{z \in \C \mid \text{Re}(z) < 0 \}.
    \label{cond: strictly stable poles}
\end{align}
In contrast to the finite-dimensional case, if $T_k$ from \eqref{eqn: closed-loop ensemble system} satisfies condition \eqref{cond: strictly stable poles}, it is \textit{not} guaranteed that system \eqref{eqn: closed-loop ensemble system} is asymptotically stable. For example, let $X = \ell^\infty$ and consider the system $\dot x(t) = Ax(t)$, where $A = \diag(-1, -1/2, -1/3, \ldots)$. Clearly, $\Sigma_{\text{disc}}(A) \subseteq H^-$ is satisfied; however, given $x(0)=(1, 1, \ldots)$, the solution is $x(t) = e^{At}x(0) = (e^{-t}, e^{-t/2}, \ldots)$, which implies that $\|x(t)\|_{\ell^\infty} = 1$ for all $t\geq 0$. Worse still, satisfying condition~\eqref{cond: strictly stable poles} does \textit{not} even guarantee that system~\eqref{eqn: closed-loop ensemble system} is \textit{stable}. Take, for example, the system $\dot x(t) = Jx(t)$ defined on $X=\ell^2$, where $J$ is an infinite-dimensional Jordan block with $-1/2$ on the diagonal. This system is shown to be unstable in \cite{azamov2023stability}. In fact, it can be shown that if $J = \diag(J_1, J_2, \cdots)$ is in the ``Jordan normal form,'' where each $J_i$ is an $n_i \times n_i$ Jordan block with $-1/2$ on the diagonal, and the $n_i$'s are not uniformly bounded, then the system $\dot x(t) = Jx(t)$ is unstable. Note that these obstructions to stability are alleviated if $T_k$ is diagonalizable, i.e., if there exists a similarity transformation that converts $T_k$ into an infinite-dimensional diagonal matrix whereby stability of system~\eqref{eqn: closed-loop ensemble system} follows immediately if \eqref{cond: stable poles} is satisfied. However, ensuring that $T_k$ is diagonalizable is challenging due to the following fact: even if the algebraic multiplicity of each $\lambda \in \Sigma_{\text{disc}}(T_k)$ is one, $T_k$ is \textit{not necessarily} diagonalizable. Motivated by this, we address in this paper the problem of feedback diagonalizability, i.e., the existence of a $k\in X^*$ such that $T_k$ is diagonalizable, in conjunction with pole placement.

%%%%%%%%%%%%%%%%%%%%%%%%%%%%%%%%%%%%%%%%%%%%%%%%%%%%%%%%%%%%
\subsection{Literature review}\label{subsec: literature}
The literature addressing feedback stabilizability of infinite ensemble systems is very sparse. The most relevant work for our purposes is~\cite{chen2025poleplacement}, where Chen provides several necessary or sufficient conditions for pole placement and feedback stabilizability for discrete linear ensemble systems of the form~\eqref{eqn: the individual systems}. Some relevant results will be reproduced shortly. Other relevant work includes~\cite{chittaro2018asymptotic}, in which the authors addressed the feedback stabilization problem for a discrete ensemble of Bloch equations, a bilinear control system. For work on feedback stabilizing \textit{finite} ensemble systems, we mention~\cite{ryan2014simultaneous, guth2025ensemble}. Beyond ensemble systems described by ordinary differential equations, we mention~\cite{alleaume2024ensembles}, where the authors addressed the problem of stabilizing an infinite ensemble of hyperbolic partial differential equations (PDEs).

Putting our problem in a broader context, we discuss a few relevant works on feedback stabilization for infinite-dimensional linear control systems $\dot x(t) = Ax(t) + Bu(t)$, where $u(t)$ is not necessarily finite-dimensional. It is well known (see, e.g.,~\cite[Chapter 6]{curtain2012introduction}) that if $(A, B)$ is exactly null controllable, then the system is (exponentially) feedback stabilizable. However, if $B$ is a compact operator (as is the case for the linear ensemble system~\eqref{eqn: infinite-dim ensemble system}), then $(A, B)$ is never exactly controllable~\cite{triggiani1975controllability}. Moreover, counterexamples have been exhibited in~\cite{triggiani1975stabilizability} which demonstrate that approximate controllability does not imply feedback stabilizability. On a positive side, it has been shown in~\cite{benchimol1977feedback} that if the semigroup generated by $A$ is similar to a contraction
semigroup, then approximate controllability implies weak feedback stabilizability. 
In~\cite{slemrod1972linear,slemrod1974note}, the author has introduced  the so-called ``$D_\epsilon$-assumption'', which requires that  the associated controllability gramian $W(\epsilon)$ is a bounded, invertible linear operator for some~$\epsilon > 0$. The author has shown that if the $D_\epsilon$-assumption and a few other conditions are satisfied, then the linear system is feedback stabilizable. However, it remains open when an $(A, B)$ pair can satisfy the assumption.

Finally, we mention the problem of \textit{simultaneous stabilization} (see, e.g.,~\cite{blondel1994simultaneous, sontag1985introduction,ghosh1985some,schonlein2026ensemble} and references therein). We point out that simultaneous stabilization is fundamentally different from ensemble feedback stabilization. The former stabilizes each individual system via a common (or parameter-dependent) feedback \textit{gain} $k$, e.g., in our discrete ensemble setting, simultaneous stabilization is achieved if there exists a scalar $k$ such that each individual closed-loop system $\dot x_n(t) = (a_n + b_n k) x_n(t)$ is asymptotically stable. In contrast, the latter stabilizes the ensemble system via a common feedback \textit{input} $u(t) = kx(t)$. This subtle difference leads to a significant distinction in closed-loop systems: under simultaneous stabilization, all individual systems are \textit{completely decoupled}, while under ensemble feedback stabilization, all systems are \textit{fully coupled} through the (full-state) common feedback input.

\subsubsection{Results from the prior work}
We recall here several key results from \cite{chen2025poleplacement}, beginning with the following theorem on the feasibility of pole placement.

\begin{theorem}[{\cite[Theorem 1]{chen2025poleplacement}}]
    Let $(a_n)\in \ell^\infty$, with $a_n > 0$ for $n\in \N$, and $(b_n)\in X$. Suppose that there exists a $k\in X^*$ such that \eqref{cond: stable poles} is satisfied; then, the following hold:
    \begin{enumerate}
        \item $(a_n)\in c_0$ and, moreover, $a_n \neq a_m$ for all $n\neq m$;
        \item $b_n \neq 0$ for all $n\in \N$.
    \end{enumerate}
    \label{Chen thm: distinct a_n's}
\end{theorem} 

\noindent Furthermore, Theorem 3 of \cite{chen2025poleplacement} says, roughly speaking, that pole placement is feasible if and only if $(a_n)$ decays at a rate of at least $O(1/n^2)$ and that the ratio $(a_n/|b_n|)$ does not grow exponentially. We will shortly see that there is a discrepancy between these decay rates and the decay rates that are sufficient for feedback stabilization.

\begin{remark}
    Suppose that pole placement is feasible; then, by item 1 of Theorem \ref{Chen thm: distinct a_n's}, $A$ is a compact operator. Since $T_k$ is a rank-one perturbation of $A$, $T_k$ is compact as well. Note that if an operator $P$ is compact, then 
    $\Sigma_{\text{ess}}(P) = \{0\}$, where $\Sigma_{\text{ess}}(P) := \Sigma(P)\setminus \Sigma_{\text{disc}}(P)$ is the essential spectrum of $P$. Consequently, if $\lambda$ is a nonzero eigenvalue of $P$, then $\lambda \in \Sigma_{\text{disc}}(P)$, and its generalized eigenspace is finite dimensional.
\end{remark}

We now turn our attention to the feedback stabilizability of system \eqref{eqn: closed-loop ensemble system} and present a constructive sufficient condition established in \cite{chen2025poleplacement}. The key to this sufficient condition is that it guarantees that $T_k$ is \textit{diagonalizable}, i.e., there exist bounded linear operators $P:X\to X$ and $Q:X \to X$ that are inverses of each other and satisfy $T_k = P\Lambda Q$, where $\Lambda=\diag(\lambda_1, \lambda_2, \ldots)$ with $\{\lambda_1, \lambda_2, \ldots\} = \Sigma_{\text{disc}}(T_k)$. In order to state the sufficient condition, we must first introduce some preliminary definitions and results.
\begin{definition}
    Let $(x_n)$ be a positive real sequence. If there exist constants $\alpha > 0$ and $r \in (0,1)$ such that 
        \begin{align*}
            \frac{x_{m}}{x_n} \leq \alpha r^{m-n} \quad \text{for all } m\geq n,
        \end{align*}
        then we say that $(x_n)$ \textbf{uniformly exponentially decays}. We will often refer to this form of decay as $\bm{(\alpha,r)}$\textbf{-exponential decay} for ease of specifying the parameters~$\alpha$ and~$r$.
    \label{defn: exponential decay}
\end{definition}
Next, we look at the particular form that the feedback $k\in X^*$ will take. First, we note that for all cases of $X$ considered in this paper, except for $\ell^\infty$ and $c$, we are able to represent $k = (k_n) \in X^*$ as 
\begin{align}
    k : x \mapsto kx = \sum_{n=1}^\infty k_n x_n \in \C.
    \label{eqn: k acting as a sum}
\end{align}
When $X = \ell^\infty$, then $X^*$ contains $\ell^1$ as a proper subspace, which implies that not every $k\in (\ell^{\infty})^*$ takes the form \eqref{eqn: k acting as a sum}. When $X = c$, then $X^* \cong \ell^1$; however, the representation of $k\in c^*$ includes an additional term $k_0$ such that $kx=\sum_{n=0}^\infty k_n x_n$, where $x_0:=\lim_{n\to \infty} x_n$. For the remainder of the paper, we will only consider feedback $k\in X^*$ that may be represented by \eqref{eqn: k acting as a sum}. We now construct the feedback law. Given two sequences $x=(x_n)$ and $y=(y_n)$, where $y$ has entries that are nonzero and satisfy $y_n\neq y_m$ for all $n\neq m$, we define the sequence $\pi(x,y) = (\pi_n(x,y))$ as follows:
\begin{align}
    \pi_n(x, y) := \lim_{N\to \infty} \prod_{\substack{m=1\\m\neq n}}^N \frac{1-x_m /y_n}{1 - y_m/y_n} = \prod_{\substack{m=1\\m\neq n}}^\infty \frac{1- x_m /y_n}{1 - y_m/y_n}, \quad \text{for }n\in \N.
    \label{defn: pi(lambda, a)}
\end{align}
Let $c_H :=  \{(\lambda_n) \mid \lambda_n \in H \text{ for } n\in \N\}$, where $H$ is the closed left half-plane. Assuming that pole placement is feasible so that item 2 of Theorem \ref{Chen thm: distinct a_n's} holds, for any null sequence $\lambda :=(\lambda_n) \in c_H$, we define the sequence of feedback gains $k(\lambda) = (k_n(\lambda))$ as follows:
\begin{align}
    k_n(\lambda) := -\frac{(a_n - \lambda_n)}{b_n} \pi_n(\lambda,a), \quad \text{for } n \in \N.
    \label{defn: feedback k}
\end{align}
Note that this form of feedback law is an infinite-dimensional extension of Ackermann's formula \cite{ackermann1972entwurf}. The following theorem guarantees that if the feedback given by \eqref{defn: feedback k} is an element of $X^*$, then $\lambda$ coincides with the poles of the closed-loop system \eqref{eqn: closed-loop ensemble system} with $k=k(\lambda)$.

\begin{theorem}[{\cite[Theorem 2]{chen2025poleplacement}}]
    Let $\lambda \in c_H$ be a null sequence and let $k(\lambda)$ be given by~\eqref{defn: feedback k}. If $k(\lambda)\in X^*$, then $ \Sigma (T_{k(\lambda)}) = \{\lambda_n \mid n \in \N \} \cup \{0\}$.
    % \vspace{-\baselineskip}
    \label{Chen thm: lambda is the spectrum of T_k}
\end{theorem}
\vspace*{-\parskip}
\noindent Note that $T_{k(\lambda)}$ is compact, so $0$ is necessarily an accumulation point of $(\lambda_n)$, and it is the only element in the essential spectrum of $T_{k(\lambda)}$. %\Sigma_{\text{disc}}(T_{k(\lambda)})=\{\lambda_n \mid n \in \N \}$ and $\Sigma_{\text{ess}}(T_{k(\lambda)})=\{0\}$. 

We now have the machinery to state the following sufficient condition for feedback stabilization, which is adapted from Theorems 4 and 5 of~\cite{chen2025poleplacement}.

% Option 1
\begin{theorem}
    Suppose that $(|b_n|)$ and $(a_n/|b_n|)$ uniformly exponentially decay, both with $\alpha=1$; then, $k(-a) \in X^*$ and $T_{k(-a)}$ is diagonalizable. Moreover, the closed-loop system $$\dot x(t) = T_{k(-a)} x(t)$$ is stable for all $X$ and asymptotically stable for all $X$ except $X=\ell^\infty$ and $X=c$.
    \label{Chen thm: 5}
\end{theorem}

%%%%%%%%%%%%%%%% MAIN RESULTS %%%%%%%%%%%%%%%%

\subsection{Main results} \label{sec: main results}
Given the findings of \cite{chen2025poleplacement}, we pose two motivating questions for our work. In the previous subsection, we saw that pole placement requires polynomial decay of $(a_n)$ and for $(a_n/|b_n|)$ not to blow up exponentially, while on the other hand, diagonalizability of $T_{k(\lambda)}$ follows if $(|b_n|)$ and $(a_n/|b_n|)$ uniformly exponentially decay. This disparity between decay rates motivates the first question we address in this paper:

\begin{enumerate}[(Q1)]
    \item Is the uniform exponential decay of $(|b_n|)$ and $(a_n/|b_n|)$ a \textit{necessary condition} for the diagonalizability of $T_k$?
\end{enumerate}
Through the following theorem, we show that the answer to (Q1) is affirmative for $(|b_n|)$ unconditionally and affirmative for $(a_n/|b_n|)$ provided the additional assumption that $(a_n)$ does not decay faster than exponentially.

%%%%%%%% MAIN RESULT #1 %%%%%%%%%%%%%%%
\begin{theorem}
    Suppose that there exists a negative real sequence $\lambda = (\lambda_n)$ such that $k(\lambda)\in X^*$ and $T_{k(\lambda)}$ is diagonalizable; then, the following hold:
    \begin{enumerate}[(C\arabic*), start=1, leftmargin=3.25em]        
        \item $(a_n)$ and $(|\lambda_n|)$ uniformly exponentially decay with $\alpha=1$.
        \label{cond: a_n and lambda_n exponential decay}
        
        \item $(|b_n|)$ uniformly exponentially decays.
        \label{cond: b_n exponential decay}
        
        \item $(g_n)$ uniformly exponentially decays, where $(g_n)$ is given by
        \begin{align}
            g_n := \frac{a_n -\lambda_n}{|b_n|}, \quad \text{for }n\in \N.
            \label{defn: g_n}
        \end{align}
        \label{cond: g_n exponential decay}
    \end{enumerate}
    Moreover, if we additionally assume that $a_{n+1}/a_n$ is uniformly bounded from below by some $r' \in (0, 1)$, then the following hold as well:
    \begin{enumerate}[(C\arabic*), start=4, leftmargin=3.25em]
        \item $(a_n/|b_n|)$ uniformly exponentially decays.
        \label{cond: a_n/b_n exponential decay}
        
        \item There exist constants $c_1, c_2 > 0$ such that $c_1 \leq |\lambda_n|/a_n \leq c_2$ for all $n\in \N$.
        \label{cond: static ratio bounds}
    \end{enumerate}
    \label{thm: nec cond}
\end{theorem}

\noindent We also note that under the pole assignment $\lambda_n = -a_n$ for $n\in \N$ (as considered in Theorem \ref{Chen thm: 5}), the uniform exponential decay of $(a_n/|b_n|)$ immediately follows from \ref{cond: g_n exponential decay}. Having answered (Q1), we now pose a natural converse question regarding the poles:

\begin{enumerate}[(Q2)]
    \item Given $(|b_n|)$ and $(a_n/|b_n|)$ uniformly exponentially decay, which conditions must the poles $(\lambda_n)$ satisfy for pole placement and feedback diagonalizability to be feasible?
\end{enumerate}
The following theorem provides a sufficient condition for $(\lambda_n)$ to satisfy, which we note includes a stronger version of \ref{cond: static ratio bounds}.

\begin{theorem}
    Suppose that $\lambda=(\lambda_n)$ is a negative real sequence, $b_n \neq 0$ for $n\in \N$, and $(a_n)$, $(b_n)$, and $(\lambda_n)$ satisfy conditions \ref{cond: a_n and lambda_n exponential decay}-\ref{cond: a_n/b_n exponential decay} and the following:
    \begin{enumerate}[(C\arabic**), start=5, leftmargin=3.5em]
        \item there exist sequences $ (\gamma_n)$, $(\omega_n)$ such that $\gamma_n \leq |\lambda_n|/a_n \leq \omega_n$ for all $n\in \N$ and, moreover, $(\omega_n-1)$ and $(1/\gamma_n - 1)$ are positive sequences in $\ell^1$;
        \label{cond: dynamic ratio bounds}
    \end{enumerate}
    then, $k(\lambda)\in X^*$ for all $X$ and $T_{k(\lambda)}$ is diagonalizable if $X=\ell^p$ for $1 \leq p \leq \infty$ or $X = c_0$.
    \label{thm: suff cond}
\end{theorem}

We conclude this section by connecting Theorem \ref{thm: suff cond} back to feedback stabilization through the following corollary, the proof of which follows the same arguments presented in the last section of the proof of \cite[Theorem 4]{chen2025poleplacement} (beginning at the bottom of page 26). Note that for the case where $X=c$, we embed $c$ into $\ell^\infty$ and apply the same arguments as in the $X=\ell^\infty$ case.

\begin{corollary}
    Suppose that the hypothesis of Theorem \ref{thm: suff cond} is satisfied; then, the closed-loop system 
    \begin{align*}
        \dot x(t) = T_{k(\lambda)} x(t)
        \label{eqn: feedback system in suff cond cor}
    \end{align*}
    is stable for all $X$ and asymptotically stable for all $X$ except $X=\ell^\infty$ and $X=c$.
    % satisfies the following:
    % \begin{enumerate}
    %     \item If $X=c_0$ or if $X = \ell^p$ for $1 \leq p < \infty$, then system \eqref{eqn: feedback system in suff cond cor} is asymptotically stable.

    %     \item If $X = \ell^\infty$ or if $X = c$, then system \eqref{eqn: feedback system in suff cond cor} is stable, but not asymptotically stable.
    % \end{enumerate}
    \label{cor: closed-loop stability}
\end{corollary}

% \begin{remark}
%     Note that while diagonalizability of $T_{k(\lambda)}$ does not necessarily hold if $X = c$, feedback stabilizability does hold since $c$ can be embedded into $\ell^\infty$, i.e., $c$ inherits the stabilizability property associated with $\ell^\infty$. 
% \end{remark}

The remainder of the paper is dedicated to the proofs of Theorems \ref{thm: nec cond} and \ref{thm: suff cond}, which are presented in Sections \S \ref{sec: proof of nec cond} and \S \ref{sec: suff thm proof}, respectively. 

% \begin{remark}
%     Recall the example system $\dot x = Ax$ with $A = \diag(-1, -1/2, \ldots )$ and $X = \ell^\infty$; this system is an example of {\color{red} item 2} in Corollary \ref{cor: closed-loop stability} since $A$ is diagonal and satisfies \eqref{cond: stable poles}, and yet, $\lim_{t\to \infty}\|x(t)\|_{\ell^\infty} \neq 0$.
% \end{remark}

%%%%%%%%%%%%%%%% PROOFS oF THEOREM 1 %%%%%%%%%%%%%%%%

\section{Proof of Theorem \ref{thm: nec cond}}\label{sec: proof of nec cond}
Suppose that there exists a negative real sequence $\lambda = (\lambda_n)$ such that $k(\lambda) \in X^*$ and $T_{k(\lambda)}$ is diagonalizable. By Theorem \ref{Chen thm: lambda is the spectrum of T_k}, it follows that $(\lambda_n) = \Sigma_{\text{disc}}(T_{k(\lambda)})$ and, as such, $T_{k(\lambda)}$ satisfies \eqref{cond: stable poles}. Hence, by Theorem \ref{Chen thm: distinct a_n's}, we have that $(a_n)\in c_0$ with pairwise distinct entries and that $(b_n)$ is nowhere zero. It can also be shown that $(\lambda_n) \in c_0$ using the same arguments that establish $(a_n)\in c_0$ (see \cite[\S 2.3.1]{chen2025poleplacement}). Since $(\lambda_n)$ is the discrete spectrum of a rank-one perturbation of a diagonal operator, it is known that the geometric multiplicity of $\lambda_n$ is one for all $n\in \N$ (see \cite[\S 3]{dobosevych2021spectra}). In a moment, we will also prove from direct computation that each $\lambda_n$ has geometric multiplicity equal to one. Then, since $T_{k(\lambda)}$ is assumed to be diagonalizable, the algebraic multiplicity of $\lambda_n$ must equal its geometric multiplicity for all $n\in \N$, which in turn implies that $(\lambda_n)$ has pairwise distinct entries. Subsequently, without loss of generality, we arrange $(a_n)$ and $(|\lambda_n|)$ in strictly monotonically decreasing order. Additionally, for ease of presentation, we assume $b_n > 0$; it will be seen that this assumption does not lose generality as $b_n$ can be replaced with $|b_n|$ in all subsequent proofs.

Due to the hypothesis that $T_{k(\lambda)}$ is diagonalizable, there exist $P,Q\in \mathcal{B}(X)$ such that $T_{k(\lambda)} = P\Lambda Q$, where $\Lambda = \diag(\lambda_1, \lambda_2, \ldots)$, and $PQ=QP=I$, where $I$ denotes the identity operator. For ease of computation, we will compute $P$ and $Q$ of a similarity transformation of $T_{k(\lambda)}$. Let $Y$ be the Banach space defined as
\begin{align*}
    Y := \{(y_n) \mid (b_ny_n) \in X \},
    % \label{defn: Y space}
\end{align*}
where for any $y = (y_n)\in Y$, we set
\begin{align*}
    \|y\|_Y := \|(b_ny_n)\|_X.
\end{align*}
We introduce the diagonal operators $B:Y \to X$, given by $(y_n) \mapsto (b_ny_n)$, and its inverse $B^{-1} : X \to Y$, given by $(x_n) \mapsto (x_n/b_n)$. Then, we define the operator $\tilde T: Y \to Y$ as
\begin{equation}
    \tilde T := B^{-1} T_{k(\lambda)} B.
    \label{defn: T tilde}
\end{equation}
For the rest of the paper, we will denote operators on $Y$ and elements of $Y$ (or $Y^*$) with a tilde, e.g., $\tilde T$. 

It is not hard to see that $\tilde T$ is diagonalizable. Let $\tilde P, \tilde Q : Y \to Y$  be defined respectively as $\tilde P:= B^{-1}PB$ and $\tilde Q := B^{-1}QB$. By the definition of $\|\cdot \|_Y$ and the fact that $P$ and $Q$ are bounded by hypothesis, it follows that $\tilde P$ and $\tilde Q$ are bounded. Given that $P$ and $Q$ are inverses of each other by hypothesis, it is easy to check that $\tilde P$ and $\tilde Q$ of are inverses of each other as well. Finally, we have
\begin{align*}
    \tilde T = B^{-1}T_{k(\lambda)}B = B^{-1}P\Lambda QB = B^{-1}PB\Lambda B^{-1}QB = \tilde P \Lambda \tilde Q,
\end{align*}
which establishes that $\tilde T$ is diagonalizable via $\tilde P$ and $\tilde Q$.

Let $\tilde k := (b_n k_n)\in Y^*$. Hence, we can rewrite $\tilde T$ as follows:
\begin{align*}
    \tilde T = A + \mathbf{1}\tilde k,
\end{align*}
where $\mathbf{1}$ denotes the vector/sequence of all ones. The left eigenvector $w_i^\top\in Y^*$ associated with eigenvalue $\lambda_i$ of $\tilde T$ is given by
\begin{align*}
    w_i^\top = ( w_{ij})_{j\in\N} = \left( \frac{(\lambda_j - a_j)\pi_j(\lambda, a)}{\lambda_i - a_j} \right)_{j\in \N},
\end{align*}
where $\pi_j(\lambda, a)$ is given by \eqref{defn: pi(lambda, a)}. Next, the right eigenvector $v_j \in Y$ associated with eigenvalue $\lambda_j$ of $\tilde T$ is given by
\begin{align*}
    v_j = ( v_{ij})_{i\in \N} = \left(\frac{a_j - \lambda_j}{a_i - \lambda_j} \right)_{i \in \N}.
\end{align*}
Because $\tilde T$ is diagonalizable, there exist sequences of scaling factors $(\sigma_n)$ and $(\tau_n)$ such that $\tilde P, \tilde Q:Y\to Y$, defined respectively as
\begin{align*}
    \tilde P : (y_n) \mapsto \left(\sum_{m=1}^\infty \frac{\sigma_m(a_m - \lambda_m) y_m}{a_n - \lambda_m} \right), \quad \tilde Q : (y_n) \mapsto \left( \sum_{m=1}^\infty \frac{\tau_n(\lambda_m - a_m)\pi_m(\lambda, a) y_m}{\lambda_n - a_m} \right),
\end{align*}
are bounded linear operators that diagonalize $\tilde T$. These two operators have the following infinite-dimensional matrix representations, respectively:
\begin{align}
\begin{split}
    \tilde P &= \left[\tilde P_{ij} := \frac{\sigma_j(a_j - \lambda_j)}{a_i - \lambda_j} \right]_{1 \leq i,j < \infty},\quad
    \tilde Q = \left[\tilde Q_{ij} := \frac{\tau_i(\lambda_j - a_j)\pi_j(\lambda, a)}{\lambda_i - a_j} \right]_{1 \leq i,j < \infty}.
    \label{defn: P and Q}
\end{split}
\end{align}
Note that $\tilde Q_i^\top = (\tilde Q_{ij})_{j\in \N}$, the $i$th row of $\tilde Q$, is the $i$th left eigenvector of $\tilde T$ and $\tilde P_j = (\tilde P_{ij})_{i\in \N}$, the $j$th column of $\tilde P$, is the $j$th right eigenvector of $\tilde T$. Finally, we note that both $\tilde P$ and $\tilde Q$ can be viewed as a variation of an infinite-dimensional Cauchy matrix, which will be advantageous for subsequent computations. 

We now turn our attention to the scaling factors $(\sigma_n)$ and $(\tau_n)$. These scaling factors satisfy important properties, as revealed in the following proposition.
\begin{proposition}
    The following hold:
    \begin{enumerate}
        \item $\pi_n(a, \lambda)$ is well defined for all $n\in \N$, where $\pi_n(a, \lambda)$ is given by \eqref{defn: pi(lambda, a)}.
        
        \item $\sigma_n \tau_n = \pi_n(a,\lambda)$ for all $n\in \N$.
        
        \item $(\sigma_n)$ is uniformly bounded from above and below.

        \item $(\pi_n(\lambda,a)\pi_n(a,\lambda))$ is uniformly bounded from above and below.
    \end{enumerate}
    \label{prop: scaling properties}
\end{proposition}
\noindent Proposition \ref{prop: scaling properties} will be critical for proving all items of Theorem \ref{thm: nec cond}. First, \ref{cond: a_n and lambda_n exponential decay} as well as \ref{cond: a_n/b_n exponential decay} and \ref{cond: static ratio bounds} will follow directly from item 4, particularly the uniform upper bound on $(\pi_n(\lambda,a)\pi_n(a,\lambda))$. We prove Proposition \ref{prop: scaling properties} in Subsection \S \ref{subsec: proof of prop: scaling properties} and establish that \ref{cond: a_n and lambda_n exponential decay}, \ref{cond: a_n/b_n exponential decay}, and \ref{cond: static ratio bounds} are necessary consequences in Subsection \S \ref{subsec: proof of C2}.

The proof of \ref{cond: b_n exponential decay} and \ref{cond: g_n exponential decay} will require more effort. The third item of Proposition \ref{prop: scaling properties} provides the first step; in particular, because $(\sigma_n)$ is uniformly bounded from below and $\tilde P$ is bounded by hypothesis, the following proposition will hold.
\begin{proposition}
    Let $(g_n)$ be the sequence given by \eqref{defn: g_n}. The following hold:
    \begin{enumerate}
        \item If $X=\ell^\infty$, $c$, or $c_0$, then there exists a constant $C > 0$ such that for all $i\in \N$,
        \begin{align}
            \sum_{j=1}^{i-1} \frac{b_i}{b_j} \leq C \quad \text{and} \quad 
            \sum_{j=i}^{2i} \frac{g_j}{g_i} \leq C.
            \label{eqn: case 1 sums/ inequalities}
        \end{align}

        \item If $X=\ell^p$ for $1\leq p < \infty$, then there exists a constant $C>0$ such that for all $j \in \N$,
        \begin{align}
            \sum_{i=1}^{j-1} \left( \frac{g_j}{g_i} \right)^p \leq C \quad \text{and}\quad
            \sum_{i = j}^{2j} \left(\frac{b_i}{b_j} \right)^p \leq C.
            \label{eqn: case 2 sums/ inequalities}
        \end{align}
    \end{enumerate}
    \label{prop: bounded sums}
\end{proposition}

The final ingredient needed to establish \ref{cond: b_n exponential decay} and \ref{cond: g_n exponential decay} is the following proposition.

\begin{proposition}
    Let $(x_n)$ be an arbitrary positive real sequence. Suppose that there exists a constant $C>0$ such that either of the following hold for all $i\in \N$:
    \begin{align}
        \sum_{j=1}^{i-1} \frac{x_i}{x_j} \leq C \quad \text{or} \quad 
        \sum_{j=i}^{2i} \frac{x_j}{x_i} \leq C.
        \label{eqn: sums/inequalities for exponential decay}
    \end{align}
    Then, $(x_n)$ uniformly exponentially decays.
    \label{prop: exponential decay}
\end{proposition}

Combining Propositions \ref{prop: bounded sums} and \ref{prop: exponential decay}, if $X$ is $\ell^\infty$, $c$, or $c_0$, then it immediately follows that $(b_n)$ and $(g_n)$ uniformly exponentially decay. If $X = \ell^p$ for $1 \leq p < \infty$, then $(b_n^p)$ and $(g_n^p)$ uniformly exponentially decay. It is not hard to see that if a sequence $(x_n^p)$ with $p >0$ is $(\alpha,r)$-exponentially decaying, then $(x_n)$ is $(\alpha^{1/p},r^{1/p})$-exponentially decaying. Hence, $(b_n)$ and $(g_n)$ uniformly exponentially decay for the second case of $X$ as well. Through these arguments, we have shown that \ref{cond: b_n exponential decay} and \ref{cond: g_n exponential decay} hold for all cases of $X$. All that is left to do is prove Propositions \ref{prop: bounded sums} and \ref{prop: exponential decay}, which we do in Subsections \S \ref{subsec: proof of prop: bounded sums} and \S \ref{subsec: proof of prop: exponential decay}, respectively.

\subsection{Proof of Proposition \ref{prop: scaling properties}}\label{subsec: proof of prop: scaling properties}
We begin with preliminaries for the proof of items 1 and 2. For $n\in \N$, to establish that $\pi_n(a,\lambda)$ is well-defined, we will first construct $\pi_n(a,\lambda;N)$, for $N\in \N$, as an approximation of $\pi_n(a,\lambda)$ such that $\pi_n(a,\lambda;N) \to \pi_n(a,\lambda)$ as $N\to \infty$. Then, we will show that $\lim_{N\to \infty} \pi_n(a,\lambda; N)$ exists for all $n\in \N$. The construction is as follows: for each $N\in \N$, let $\pi(a, \lambda; N) = (\pi_n(a, \lambda; N))$ be an eventually zero sequence defined as follows:
\begin{align}
    \pi_n(a, \lambda; N) &:= \begin{cases}
        \prod_{m=1, m\neq n}^N \frac{1 - a_m / \lambda_n}{1 - \lambda_m/\lambda_n} \quad &\text{if } 1 \leq n \leq N \\
        0 &\text{if } n\geq N+1
    \end{cases}, \quad \text{for }n\in \N.
    \label{defn: N approx of pi sequence}
\end{align}
Note that by definition \eqref{defn: pi(lambda, a)}, $\pi_n(a,\lambda) = \lim_{N\to \infty}\pi_n(a, \lambda; N)$, if the limit exists. We also let $\pi(\lambda,a; N) = (\pi_n(\lambda,a;N))$ for $N\in \N$, where $\pi_n(\lambda,a;N)$ is given by \eqref{defn: N approx of pi sequence} but with $a$ and $\lambda$ swapped. For $N\in \N$, we construct $\tilde Q(N) : Y \to Y$, an approximation of $\tilde Q$, which is defined the same way as $\tilde Q$ in \eqref{defn: P and Q} except with $\pi_j(\lambda,a;N)$ replacing $\pi_j(\lambda,a)$:
\begin{align*}
    \tilde Q(N) = \left[ \tilde Q_{ij}(N) := \frac{\tau_i(\lambda_j - a_j) \pi_j(\lambda, a; N)}{\lambda_i - a_j} \right]_{1\leq i,j < \infty}.
\end{align*}
Note that $\tilde Q_i^\top (N) \in Y^*$, the $i$th row of $\tilde Q(N)$, is a sequence with only $N$ nonzero entries, i.e.,
\begin{align*}
    \tilde Q_i^\top (N) = (\tilde Q_{i1}(N), \ldots, \tilde Q_{iN}(N), 0, 0, \ldots).
\end{align*}
In order to establish items 1 and 2, we will need to show that $\tilde Q_i^\top(N) \to \tilde Q_i^\top$ as $N\to \infty$ for all $i\in \N$. For all cases of $X$ except $X=\ell^1$, we establish the strong convergence of $\tilde Q_i^\top(N)$, i.e., $\|\tilde Q_i^\top - \tilde Q_i^\top(N)\|_{Y^*} \to 0$. When $X=\ell^1$, we show that weak convergence holds, i.e., for all $y\in Y$, $\tilde Q_i^\top(N) y \to \tilde Q_i^\top y$. These convergence results are summarized in the lemma below.
\begin{lemma}[Convergence of $\tilde Q_i^\top(N)$]
    The following hold:
    \begin{enumerate}
        \item If $X=c$, $c_0$, or $\ell^p$ for $1 < p \leq \infty$, then for all $i\in \N$,
        \begin{align*}
            \lim_{N\to \infty} \| \tilde Q_i^\top - \tilde Q_i^\top (N)\|_{Y^*} = 0.
        \end{align*}

        \item If $X = \ell^1$, then for all $y\in Y$ and for all $i\in \N$,
        \begin{align*}
            \lim_{N\to \infty} |\tilde Q_i^\top y - \tilde Q_i^\top(N) y|=0.
        \end{align*}
    \end{enumerate}
    \label{lem: convergence of Q(N)}
\end{lemma}
The proof of Lemma \ref{lem: convergence of Q(N)} takes some effort, so we will save it for \S\ref{subsubsec: convegence of Q(N)} at the end of this subsection. Moving forward with the proof of Proposition \ref{prop: scaling properties}, we establish below items 1 and 2 concurrently since these two results are intertwined.

\begin{proof}[Proof of items 1 and 2.]
    We begin by introducing the $N\times N$ matrices $\tilde Q'(N)$ and $\tilde P'$, which are obtained by truncating $\tilde Q(N)$ and $\tilde P$ as follows: 
    \begin{align*}
        \tilde Q'(N) := [\tilde Q_{ij}(N)]_{1 \leq i,j \leq N} \quad \text{and} \quad
        \tilde P' := [\tilde P_{ij}]_{1\leq i,j \leq N}.
    \end{align*}
    By construction, we have that
    \begin{align}
        \tilde Q_i^\top(N) \tilde P_i = \tilde Q_i'^\top(N) \tilde P_i', \quad \text{for all } 1\leq i \leq N.
        \label{eqn: Q(N)P is equal to Q prime times P prime}
    \end{align}
    Next, we highlight that $\tilde Q'(N)$ and $\tilde P'$ are Cauchy-like matrices, which is advantageous since the inverse of a Cauchy matrix has an explicit, closed-form expression. To elaborate on this point, consider the following $N \times N$ Cauchy matrix
    \begin{align}
        G:=\left[\frac{1}{a_i - \lambda_j} \right]_{1\leq i, j \leq N},
        \label{eqn: Cauchy matrix}
    \end{align}
    where the $a_i$'s and $\lambda_j$'s are distinct, but otherwise arbitrary, complex numbers. Then, the inverse of $G$ is given by 
    \begin{align}
            G^{-1} = \left[G_{ij}^{-1} = \frac{(\lambda_i - a_i)(a_j - \lambda_j) \pi_i(a, \lambda; N)\pi_j(\lambda, a; N)}{\lambda_i - a_j} \right]_{1 \leq i,j \leq N}.
            \label{eqn: Cauchy inverse}
    \end{align}
    For the proof of \eqref{eqn: Cauchy inverse}, see \cite{schechter1959inversion} for details. Combining equation \eqref{eqn: Q(N)P is equal to Q prime times P prime} with the Cauchy inverse formula \eqref{eqn: Cauchy inverse}, it follows that 
        \begin{align*}
        \pi_i(a, \lambda; N) = \frac{\sigma_i \tau_i}{\tilde Q_i'^\top(N) \tilde P'_i} = \frac{\sigma_i \tau_i}{\tilde Q_i^\top(N) \tilde P_i}, \quad \text{for } 1\leq i\leq N.
        \end{align*}
    Taking the limit as $N\to \infty$ yields
    \begin{align*}
        \lim_{N\to \infty} \pi_i(a, \lambda; N) = \frac{\sigma_i \tau _i}{\lim_{N\to \infty} \tilde Q_i^\top(N) \tilde P_i} = \frac{\sigma_i\tau_i}{\tilde Q_i^\top \tilde P_i} = \sigma_i \tau_i,
    \end{align*}
    where the second to last equality follows from Lemma \ref{lem: convergence of Q(N)} and the last inequality follows from the fact that $\tilde P$ and $\tilde Q$ are inverses of each other by hypothesis. Recall that $\pi_i(a, \lambda; N) \to \pi_i(a,\lambda)$ as $N\to \infty$ by definition. Thus, we have shown that $\pi_i(a,\lambda) = \sigma_i \tau_i$ for all $i\in \N$, which completes the proof of items 1 and 2.
\end{proof}
We now prove items 3 and 4.

\begin{proof}[Proof of item 3.]
    First, since $\tilde P$ is bounded by assumption and its diagonal entries are given by $\tilde P_{ii} = \sigma_i$ for all $i\in \N$, it follows that $(\sigma_n)$ is uniformly bounded from above. We now prove that $(\sigma_n)$ is uniformly bounded from below. Multiplying the equation from item 2 on both sides by $\pi_i(\lambda,a)$, we have $\sigma_i \tau_i \pi_i(\lambda,a) = \pi_i(\lambda,a) \pi_i(a, \lambda)$, which in turn yields
    \begin{align*}
      |\sigma_i|^{-1} = \frac{|\tau_i \pi_i(\lambda, a)|}{|\pi_i(\lambda, a) \pi_i(a, \lambda)|}, \quad \text{for all } i\in \N.
    \end{align*}
    Since the $ii$th entry of $\tilde Q$ is $\tilde Q_{ii} = \tau_i \pi_i(\lambda,a)$ and $\tilde Q$ is bounded by assumption, it follows that $(\tau_n \pi_n(\lambda,a))$ is uniformly bounded above. Next, we argue that $|\pi_n(\lambda, a)\pi_n(a, \lambda)| \geq 1$ for all $n\in \N$. For notational convenience, we define for $n\in \N$,
    \begin{align}
        \begin{split}
            Z^-_n &:= \prod_{m=1}^{n-1} \frac{(1 - \lambda_m/a_n)(1-a_m/\lambda_n)}{(1-a_m/a_n)(1-\lambda_m/\lambda_n)}, \quad
            Z^+_n := \prod_{m=n+1}^\infty \frac{(1 - \lambda_m/a_n)(1-a_m/\lambda_n)}{(1-a_m/a_n)(1-\lambda_m/\lambda_n)}.
            \label{defn: Z_n plus, Z_n minus}
        \end{split}
    \end{align}
    Note that $|\pi_n(\lambda,a)\pi_n(a,\lambda)| = |Z_n^-Z_n^+|$ for all $n\in \N$. We show that $|Z_n^-|\geq1$ and $|Z_n^+|\geq 1$. Recall that $\lambda_n<0$ for all $n\in \N$, so, e.g., $1-\lambda_m/a_n = 1 + |\lambda_m|/a_n$ for all $m,n\in \N$. For $m<n$, we have $a_m/a_n > 1$ and $\lambda_m/\lambda_n > 1$ because $(a_n)$ and $(|\lambda_n|)$ monotonically decrease. Hence,
    \begin{align*}
        |Z^-_n| = \prod_{m=1}^{n-1} \frac{(1 + |\lambda_m|/a_n)(1+a_m/|\lambda_n|)}{|1 -a_m/a_n||1-\lambda_m/\lambda_n|}
        &\geq \prod_{m=1}^{n-1} \frac{(1 + |\lambda_m|/a_n)(1 +a_m/|\lambda_n|)}{|a_m/a_n||\lambda_m/\lambda_n|} \\
        &\geq \prod_{m=1}^{n-1} \frac{|\lambda_m/a_n||a_m/\lambda_n|}{|a_m/a_n||\lambda_m/\lambda_n|} = 1.
    \end{align*}
    For the $m>n$ case, we have $0<a_m/a_n <1$ and $0<\lambda_m/\lambda_n < 1$. Hence, the denominator satisfies $0<|1 - a_m/a_n||1 - \lambda_m/\lambda_n|<1$. Since the numerator of the argument is always greater than one, it follows that $|Z_n^+| \geq 1$. 
    
    Combining the above arguments, we have that
    \begin{align*}
        |\sigma_i|^{-1} = \frac{|\tau_i \pi_i(\lambda,a)|}{|\pi_i(\lambda, a) \pi_i(a, \lambda)|} \leq |\tau_i \pi_i(\lambda,a)| \leq \sup_{i\in \N}|\tau_i \pi_i(\lambda,a)| < \infty.
    \end{align*}
    This shows that $(\sigma_n)$ is uniformly bounded from below, completing the proof of item 3.
\end{proof}
\begin{proof}[Proof of item 4.] We have already shown that $|\pi_n(\lambda,a)\pi_n(a, \lambda)| \geq 1$ for all $n\in \N$ in the previous proof. We now establish the uniform upper bound, which follows immediately from prior results. From item 2, we have $\pi_n(\lambda, a) \pi_n(a, \lambda) = \sigma_n \tau_n \pi_n(\lambda, a)$ for all $n\in \N$. Then, since $(\sigma_n)$ and $(\tau_n\pi_n(\lambda,a))$ are uniformly bounded above, as established in the previous proof, it follows that $(\pi_n(\lambda,a)\pi_n(a,\lambda))$ is uniformly bounded above.
\end{proof}
We conclude this subsection with the proof of Lemma \ref{lem: convergence of Q(N)}.
\subsubsection{Proof of Lemma \ref{lem: convergence of Q(N)}}\label{subsubsec: convegence of Q(N)}
\begin{proof}[Proof of item 1.]
    Suppose that $X=c$, $c_0$, or $\ell^p$ for $1 < p \leq \infty$. Let $Q_i^\top := (\tilde Q_{ij}/b_j)_{j\in \N}$. Since $\tilde Q_i^\top \in Y^*$, it follows that $ Q_i^\top \in \ell^q$, where $1 \leq q < \infty$ satisfies $1/p + 1/q = 1$. We show that for any $\epsilon > 0$, there exists a positive integer $N_{\epsilon}$ such that
    \begin{align*}
        \| \tilde Q_i^\top - \tilde Q_i^\top(N) \|_{Y^*} = \| Q_i^\top - Q_i^\top(N) \|_{X^*} < \epsilon, \quad \text{ for all } N\geq N_\epsilon.
    \end{align*}
    Because $Q_i^\top \in \ell^q$, for any $\epsilon>0$, there exists a positive integer $N'$ such that
    \begin{align}
        \sum_{j=N' + 1}^\infty |Q_{ij}|^q < \frac{\epsilon^q}{2}.
        \label{eqn: existence of N prime}
    \end{align}
    Next, for $n \leq N$, we define
    \begin{align}
        r_n(N) := \prod_{m = N + 1}^\infty \frac{1 - a_m / a_n}{1 - \lambda_m / a_n}.
        \label{defn: r_n(N)}
    \end{align}
    Note that $\pi_n(\lambda, a; N) = r_n(N) \pi_n(\lambda, a)$ for all $n\leq N$; hence, $Q_{ij}(N) = r_j(N) Q_{ij}$ for all $j\leq N$. Next, since $(a_n)$ monotonically decreases, $0 < 1 - a_m/ a_n < 1$ holds for $m > n$. It follows that
    \begin{align*}
        |r_n(N)| = \left|\frac{\pi_n(\lambda, a; N)}{\pi_n(\lambda, a)}\right| \leq 1, \quad \text{ for } n \leq N.
    \end{align*}
    An immediate consequence is that 
    \begin{align}
        |Q_{ij} - Q_{ij}(N)| = \left[1 - \frac{\pi_j(\lambda, a; N)}{\pi_j(\lambda, a)} \right] |Q_{ij}|  \leq|Q_{ij}|, \quad \text{ for } j\leq N.
        \label{eqn:  Q tilde - Q(N) tilde is less than Q tilde}
    \end{align}
    Moreover, since $ \pi_n(\lambda, a; N) \to \pi_n(\lambda, a)$ as $N\to \infty$ by definition, it follows that $r_n(N) \to 1$ as $N\to \infty$. Hence, given the integer $N'$ satisfying \eqref{eqn: existence of N prime}, there exists an integer $N_\epsilon \geq N'$ such that
    \begin{align}
        0 < \left[1 - \frac{\pi_j(\lambda, a; N)}{\pi_j(\lambda, a)} \right]^q \leq \frac{\epsilon^q}{2 \|Q_i^\top\|_{\ell^q}^q}, \quad \text{ for all } j\leq N' \text{ and for all } N \geq N_\epsilon.
        \label{eqn: existence of N epsilon}
    \end{align}
    Putting everything together, we have for all $N \geq N_\epsilon$,
    \begin{align*}
        \|Q_i^\top - Q_i^\top (N)\|_{\ell^q}^q &= \sum_{j = 1}^{N'} |Q_{ij} - Q_{ij} (N)|^q + \sum_{j = N'+1}^N |Q_{ij} - Q_{ij} (N)|^q + \sum_{j = N+1}^\infty |Q_{ij}|^q \\
        &\leq \sum_{j = 1}^{N'} \left[1 - \frac{\pi_j(\lambda, a; N)}{\pi_j(\lambda, a)}\right]^q |Q_{ij}|^q + \sum_{j = N'+1}^\infty |Q_{ij}|^q \\
        &\leq \frac{\epsilon^q}{2 \|Q_i^\top \|_{\ell^q}^q} \sum_{j=1}^{N'} |Q_{ij}|^q + \sum_{j = N'+1}^\infty |Q_{ij}|^q \\
        &< \frac{\epsilon^q}{2 \|Q_i^\top \|_{\ell^q}^q} \sum_{j=1}^{N'} |Q_{ij}|^q + \frac{\epsilon^q}{2} < \frac{\epsilon^q}{2} + \frac{\epsilon^q}{2} = \epsilon^q,
    \end{align*}
    where the first inequality follows from \eqref{eqn:  Q tilde - Q(N) tilde is less than Q tilde}, the second inequality from \eqref{eqn: existence of N epsilon}, and the third inequality from \eqref{eqn: existence of N prime}. This completes the proof of item 1.
    
    \medskip
    \noindent \textit{Proof of item 2.} Suppose that $X=\ell^1$. Then, $ Q_i^\top \in \ell^\infty$, where $Q_i^\top$ is defined the same as above in the proof of item 1. Given an arbitrary $y\in Y$, let $x := (b_n y_n) \in X$. Clearly, we have $\tilde Q_i^\top y = Q_i^\top x$. Since $x \in \ell^1$, for any $\epsilon>0$, there exists a positive integer $N'$ such that 
    \begin{align}
        \sum_{j=N'+1}^\infty |x_j| < \frac{\epsilon}{2\|Q_i^\top \|_{\ell^\infty}}.
        \label{eqn: tail bound on P tilde}
    \end{align}
    As in the first case, we have $Q_{ij}(N) = r_j(N) Q_{ij}$ for all $j \leq N$, where $r_j(N)$ is given by~\eqref{defn: r_n(N)}. Since $\lim_{N\to \infty} r_j(N)=1$, given the integer $N'$ satisfying \eqref{eqn: tail bound on P tilde}, there exists an integer $N_\epsilon \geq N'$ such that 
    \begin{align}
        0< 1 - \frac{Q_{ij}(N)}{Q_{ij}} \leq \frac{\epsilon}{2 \|x\|_{\ell^1} \|Q_i^\top\|_{\ell^\infty}}, \quad \text{for all } j\leq N' \text{ and for all } N\geq N_\epsilon. 
        \label{eqn: N epsilon for case 2}
    \end{align}
    Therefore, it follows that for all $N \geq N_{\epsilon}$,
    \begin{align*}
        |Q_i^\top x - Q_{ij}(N)^\top x| &= \left|\sum_{j=1}^\infty \left[Q_{ij} - Q_{ij}(N) \right]x_j \right| \\
        &\leq \sum_{j=1}^\infty |Q_{ij}| \left[1 - \frac{Q_{ij}(N)}{Q_{ij}}\right] |x_j| + \sum_{j=N'+1}^\infty |Q_{ij} - Q_{ij}(N)||x_j| \\
        &\leq \|Q_i^\top\|_{\ell^\infty} \sum_{j=1}^{N'} \left[1 - \frac{Q_{ij}(N)}{Q_{ij}}\right] |x_j| + \|Q_i^\top\|_{\ell^\infty}\sum_{j=N'+1}^\infty |x_j| \\
        &< \frac{\epsilon}{2 \|x\|_{\ell^1}} \sum_{j=1}^{N'} |x_j| + \frac{\epsilon}{2} < \frac{\epsilon}{2} + \frac{\epsilon}{2} = \epsilon,
    \end{align*}
    where the second inequality follows from \eqref{eqn:  Q tilde - Q(N) tilde is less than Q tilde} and from the fact that $Q_{ij}(N)=0$ for $j> N$ by definition, and the third inequality follows from \eqref{eqn: tail bound on P tilde} and \eqref{eqn: N epsilon for case 2}. This completes the proof.
\end{proof}

%%%%%%%%%%%%%%%%%%%%%%%%%%%%%%%%%%%%%%%%%%%%%%%%%%%%%%%%%%%%%%
\subsection{Proof of \ref{cond: a_n and lambda_n exponential decay}, \ref{cond: a_n/b_n exponential decay}, and \ref{cond: static ratio bounds}}\label{subsec: proof of C2}

In this subsection, we show that these conditions necessarily follow from the fact that $(\pi_n(\lambda,a) \pi_n(a,\lambda))$ uniformly bounded above.
\begin{proof}[Proof of \ref{cond: a_n and lambda_n exponential decay}.]
    Recall that $\pi_n(\lambda,a)\pi_n(a,\lambda) = Z_n^-Z_n^+$, where $Z_n^-$ and $Z_n^+$ are given by \eqref{defn: Z_n plus, Z_n minus}. Since $(Z_n^-Z_n^+)\in \ell^\infty$ by Proposition \ref{prop: scaling properties} and $|Z_n^-| \geq 1$ for all $n\in \N$, which was shown in the proof of item 3 of Proposition \ref{prop: scaling properties}, it follows that $(Z_n^+)\in \ell^\infty$. Since $0<a_m/a_n < 1$ and $0 <\lambda_m/\lambda_n < 1$ for $m>n$ due to monotonicity, we have
    \begin{align*}
         |Z^+_n| = \prod_{m=n+1}^\infty \frac{(1 + |\lambda_m|/a_n)(1+a_m/|\lambda_n|)}{(1 -a_m/a_n)(1-\lambda_m/\lambda_n)} \geq \prod_{m=n+1}^\infty \frac{1}{1 - a_m/a_n}.
    \end{align*}
    For the sake of contradiction, assume that there does not exist an $r\in (0,1)$ such that $a_m/a_n \leq r^{m-n}$ for $m\geq n$. Consequently, there exists a subsequence $(a_{n_k})_{k\in \N}$ such that $\lim_{k\to \infty} a_{n_k + 1}/a_{n_k} = 1$. This implies that 
    \begin{align*}
        \lim_{k\to \infty} \frac{1}{1 - a_{n_k + 1}/a_{n_k}} = \infty,
    \end{align*}
    which contradicts $(Z_n^+)\in \ell^\infty$. Therefore, there exists $r\in (0,1)$ such that $a_m/a_n \leq r^{m-n}$ for $m\geq n$, i.e., $(a_n)$ uniformly exponentially decays with $\alpha=1$. Next, since we also have
    \begin{align*}
         |Z^+_n| = \prod_{m=n+1}^\infty \frac{(1 + |\lambda_m|/a_n)(1+a_m/|\lambda_n|)}{(1 -a_m/a_n)(1-\lambda_m/\lambda_n)} \geq \prod_{m=n+1}^\infty \frac{1}{1 - \lambda_m/\lambda_n},
    \end{align*}
    the same arguments as before show that $(|\lambda_n|)$ uniformly exponentially decays with $\alpha=1$. This completes the proof.
\end{proof}
For the following two proofs, we assume that there exists an $r' \in (0,1)$ satisfying $a_{n+1}/a_n \geq r'$ for all $n\in \N$. We begin with the proof of \ref{cond: static ratio bounds}. Then, the proof of \ref{cond: a_n/b_n exponential decay} is easily established from \ref{cond: g_n exponential decay} and \ref{cond: static ratio bounds}.

\begin{proof}[Proof of \ref{cond: static ratio bounds}.]
    For the sake of contradiction, suppose that there does not exist a $c_2>0$ such that $|\lambda_n|/a_n \leq c_2$ for all $n\in \N$. Recall from the previous proof that $(Z_n^+) \in \ell^\infty$. Again using the monotonicity of $(a_n)$ and $(|\lambda_n|)$, we have
    \begin{align*}
         |Z^+_n| = \prod_{m=n+1}^\infty \frac{(1 + |\lambda_m|/a_n)(1 + a_m/|\lambda_n)|}{|1-a_m/a_n||1-\lambda_m/\lambda_n|} \geq \prod_{m=n+1}^\infty (1 + |\lambda_m|/a_n) 
        > 1 + r'|\lambda_{n+1}|/a_{n+1}.
    \end{align*}
    Hence, we arrive at the contradiction that $(Z_n^+)\in \ell^\infty$ but $|\lambda_{n+1}|/a_{n+1}$ is not bounded above, which establishes that $c_2$ must exist. Next, since it also holds that $$|Z_n^+| \geq \prod_{m=n+1}^\infty (1 + a_m/|\lambda_n|) > 1 + r'a_n/|\lambda_n|,$$ it follows that there exists a $c_1>0$ satisfying $c_1 \leq |\lambda_n|/a_n$ for all $n\in \N$ using similar arguments as above. This completes the proof.
\end{proof}

\begin{proof}[Proof of \ref{cond: a_n/b_n exponential decay}.]
    From \ref{cond: g_n exponential decay}, there exist $\alpha>0$, $r\in (0,1)$ such that $(g_n)$ is $(\alpha,r)$-exponentially decaying. Then, for all $m\geq n$,
    \begin{align*}
        \frac{a_m}{b_m} < \frac{a_m + |\lambda_m|}{b_m} \leq \left(\frac{a_n + |\lambda_n|}{b_n}\right) \alpha r^{m-n} \leq \left(\frac{a_n}{b_n} \right)(1 + c_2)\alpha r^{m-n},
    \end{align*}
    where the last inequality follows from \ref{cond: static ratio bounds}. This completes the proof.
\end{proof}

%%%%%%%%%%%%%%%%%%%%%%%%%%%%%%%%%%%%%%%%%%%%%%%%%%%%%%%%%%%%%%%%%%%%%%

\subsection{Proof of Proposition \ref{prop: bounded sums}}\label{subsec: proof of prop: bounded sums}
Let $y := (x_n/b_n) \in Y$ for an arbitrary sequence $(x_n) \in X$. We define
\begin{align*}
    y' := \tilde P y = \left( \sum_{j=1}^\infty \frac{(a_j - \lambda_j)\sigma_jx_j}{(a_i - \lambda_j)b_j} \right)_{i\in \N}.
\end{align*}
Since $y'\in Y$, it follows that $(b_iy'_i)_{i\in \N} \in X$. We now consider the two different cases of $X$.

\begin{proof}[Proof of item 1.]
Let $X= \ell^\infty$, $c$, or $c_0$. For an arbitrary fixed $i\in \N$, we choose $(x_n)\in X$ to be defined as follows
\begin{align*}
    x_n := \begin{cases}
        \sgn(\sigma_n), \quad &\text{for } 1\leq n \leq 2i\\
        0, \quad &\text{for } n > 2i
    \end{cases},
    \quad \text{for } n\in \N.
\end{align*}
Recall from item 3 of Proposition \ref{prop: scaling properties} that there exists a constant $L > 0$ such that $|\sigma_n| \geq L$ for all $n\in \N$. Hence, we have
\begin{align*}
    |b_i y'_i| &= \left|\sum_{j=1}^\infty \frac{(a_j - \lambda_j)\sigma_j x_j/b_j}{(a_i - \lambda_j)/b_i} \right| = \sum_{j=1}^{2i} \frac{(a_j + |\lambda_j|)|\sigma_j|/b_j}{(a_i + |\lambda_j|)/b_i} \geq L \sum_{j=1}^{2i} \frac{(a_j + |\lambda_j|)/b_j}{(a_i + |\lambda_j|)/b_i}.
\end{align*}
Since $(b_i y'_i) \in X$, we have $|b_i y'_i| \leq \|(b_iy'_i)\|_{\ell^\infty} < \infty$.  It follows that
\begin{align*}
    \frac{1}{L}\|(b_iy'_i)\|_{\ell^\infty} \geq \sum_{j=1}^{2i} \frac{(a_j + |\lambda_j|)/b_j}{(a_i + |\lambda_j|)/b_i} = \sum_{j=1}^{i-1} \frac{(a_j + |\lambda_j|)/b_j}{(a_i + |\lambda_j|)/b_i} + \sum_{j=i}^{2i} \frac{(a_j + |\lambda_j|)/b_j}{(a_i + |\lambda_j|)/b_i}.
\end{align*}
We now look at the two summations on the right hand side individually. First, considering the sum over $1\leq j < i$, we have $a_j > a_i$ because $(a_n)$ monotonically decreases. Hence,
\begin{align*}
    \frac{1}{L}\|(b_iy'_i)\|_{\ell^\infty} \geq \sum_{j=1}^{i-1} \frac{(a_j + |\lambda_j|)/b_j}{(a_i + |\lambda_j|)/b_i} > \sum_{j=1}^{i-1} \frac{(a_i + |\lambda_j|)/b_j}{(a_i + |\lambda_j|)/b_i} = \sum_{j=1}^{i-1} \frac{b_i}{b_j},
\end{align*}
which establishes the first inequality of \eqref{eqn: case 1 sums/ inequalities}. Next, for the sum over $i\leq j \leq 2i$, we have $|\lambda_j| \leq |\lambda_i|$ again due to monotonicity. It follows that
\begin{align*}
    \frac{1}{L}\|(b_iy'_i)\|_{\ell^\infty} \geq \sum_{j=i}^{2i} \frac{(a_j + |\lambda_j|)/b_j}{(a_i + |\lambda_j|)/b_i} \geq \sum_{j=i}^{2i} \frac{(a_j + |\lambda_j|)/b_j}{(a_i + |\lambda_i|)/b_i} = \sum_{j=i}^{2i} \frac{g_j}{g_i},
\end{align*}
where the last equality follows from definition \eqref{defn: g_n}. This establishes the second inequality of \eqref{eqn: case 1 sums/ inequalities}, which completes the proof of item 1.
\end{proof}

\begin{proof}[Proof of item 2.] 
Let $X=\ell^p$, where $1 \leq p < \infty$. For this case of $X$, there exists a constant $C>0$ such that
\begin{align*}
    C \geq \sum_{i=1}^\infty |b_iy'_i|^p = \sum_{i=1}^\infty\left|\sum_{j=1}^\infty \frac{(a_j + |\lambda_j|)\sigma_j x_j/b_j}{(a_i + |\lambda_j|)/b_i} \right|^p.
\end{align*}
Let $(x_n) = e_k$ for some $k\in \N$, where $e_k\in X$ denotes the $k$th unit basis vector/sequence, i.e., the sequence of all zeros except for $1$ in the $k$th entry. Under this choice of $(x_n)$,

\begin{align*}
    \left|\sum_{j=1}^\infty \frac{(a_j + |\lambda_j|)\sigma_j (e_k)_j/b_j}{(a_i + |\lambda_j|)/b_i} \right|^p = \left|\frac{(a_k + |\lambda_k|)\sigma_k/b_k}{(a_i + |\lambda_k|)/b_i} \right|^p \geq L^p\left(\frac{(a_k + |\lambda_k|)/b_k}{(a_i + |\lambda_k|)/b_i} \right)^p.
\end{align*}
It follows that for all $k\in \N$,
\begin{align*}
    \frac{C}{L^p} \geq \sum_{i=1}^\infty \left(\frac{(a_k + |\lambda_k|)/b_k}{(a_i + |\lambda_k|)/b_i} \right)^p \geq  \sum_{i=1}^{k-1} \left(\frac{(a_k + |\lambda_k|)/b_k}{(a_i + |\lambda_k|)/b_i} \right)^p + \sum_{i=k}^{2k} \left(\frac{(a_k + |\lambda_k|)/b_k}{(a_i + |\lambda_k|)/b_i} \right)^p.
\end{align*}
For notational consistency with the statement of item 2, though with abuse of index variable, we replace $k$ with $j$. Now, considering the sum over $1 \leq i < j$, since it holds here that $|\lambda_i| > |\lambda_j|$, we have
\begin{align*}
    \frac{C}{L^p} \geq\sum_{i=1}^{j-1} \left(\frac{(a_j + |\lambda_j|)/b_j}{(a_i + |\lambda_j|)/b_i}\right)^p > \sum_{i=1}^{j-1} \left( \frac{(a_j + |\lambda_j|)/b_j}{(a_i + |\lambda_i|)/b_i} \right)^p = \sum_{i=1}^{j-1} \left( \frac{g_j}{g_i} \right)^p,
\end{align*}
which establishes the first inequality of \eqref{eqn: case 2 sums/ inequalities}. Next, for the sum over $j\leq i \leq 2j$, since it holds that $a_i\leq a_j$, we have
\begin{align*}
     \frac{C}{L^p} \geq \sum_{i=j}^{2j} \left(\frac{(a_j + |\lambda_j|)/b_j}{(a_i + |\lambda_j|)/b_i} \right)^p \geq \sum_{i=j}^{2j} \left( \frac{(a_i + |\lambda_j|)/b_j}{(a_i + |\lambda_j|)/b_i}\right)^p = \sum_{i=j}^{2j} \left(\frac{b_i}{b_j} \right)^p.
\end{align*}
This establishes the second inequality of \eqref{eqn: case 2 sums/ inequalities} and, thus, completes the proof of item 2.
\end{proof}
%%%%%%%%%%%%%%%%%%%%%%%%%%%%%%%%%%%%%%%%%%%%%%%%%%%%%%%%%%%%%%%%%%%

\subsection{Proof of Proposition \ref{prop: exponential decay}}\label{subsec: proof of prop: exponential decay}

\begin{proof}[Proof.]
    We first assume that there exists a $C>0$ such that the first inequality of \eqref{eqn: sums/inequalities for exponential decay} holds for all $i\in \N$. We define the set $\mathcal{S}_i \subset \N$, for $i\in \N$, as follows:
    \begin{align*}
        \mathcal{S}_i := \{j\in \N \mid 1\leq j<i \text{ and } x_j/x_i < 2 \}.
    \end{align*}
    Note that $\mathcal{S}_i$ may be empty. Let $K_i$ be given by 
    \begin{align*}
        K_i := \inf \mathcal{S}_i.
    \end{align*}
    We show that the difference $(i - K_i)$ is uniformly bounded from above. If $\mathcal{S}_i$ is empty for some $i\in \N$, then $K_i=\infty$, so $(i-K_i)$ is trivially upper bounded by any constant. Suppose that $\mathcal{S}_i$ is not empty. By definition, $x_{K_i}/x_i < 2$. Next, for all $K_i \leq j \leq i-1$, we have $x_j/x_{K_i} \leq C$ by hypothesis. Combining these two inequalities yields
    \begin{align*}
        x_i > \frac{x_{K_i}}{2} \geq \frac{x_j}{2C}, \quad \text{for } K_i \leq j \leq i-1.
    \end{align*}
    It follows that
    \begin{align*}
        C\geq \sum_{j=1}^{i-1} \frac{x_i}{x_j} \geq \sum_{j = K_i}^{i-1} \frac{x_i}{x_j} > \sum_{j=K_i}^{i-1} \frac{1}{2 C} = \frac{1}{2 C} (i - K_i),
    \end{align*}
    which that implies $(i - K_i)$ is uniformly bounded from above by $2C^2$. For convenience, we define $\delta:= \lceil2C^2 + 1 \rceil$, where $\lceil \cdot \rceil$ denotes the ceiling function. By definition, it is guaranteed that $x_{n+\delta}/x_n \leq 1/2$ for all $n\in \N$. Next, given $m\geq n$, there exist integers $k\geq 0$ and $0 \leq s \leq \delta -1$ such that $m-n = k\delta + s$. Then, we have
    \begin{align*}
        \frac{x_m}{x_n} = \frac{x_{n + k\delta + s}}{x_n} = \frac{x_{n+s}}{x_n} \cdot \frac{x_{n+k\delta + s}}{x_{n+s}} \leq C \left(\frac{1}{2} \right)^k.
    \end{align*}
    Substituting in $k=(m-n-s)/\delta$,
    \begin{align*}
        \left(\frac{1}{2} \right)^{\frac{m-n-s}{\delta}} = \left( 2^{-1/\delta}\right)^{m-n}2^{s/\delta} \leq \left( 2^{-1/\delta}\right)^{m-n}2^{(\delta - 1)/\delta} \leq \left( 2^{-1/\delta}\right)^{m-n}2^\delta.
    \end{align*}
    Hence, we have shown that $x_m/x_n \leq \alpha r^{m-n}$, where $\alpha = C2^\delta$ and $r = 2^{-1/\delta}$. This completes the proof given that the first inequality of \eqref{eqn: sums/inequalities for exponential decay} holds.

    \medskip
    \noindent Now, let us assume that the second inequality of \eqref{eqn: sums/inequalities for exponential decay} holds for some $C>0$ and for all $i\in \N$. For each $i\in \N$, we define the set $\mathcal{S}'_i \subset \N$ as
    \begin{align*}
        \mathcal{S}'_i := \{j\in \N \mid i \leq j \leq 2i \text{ and } x_j/x_i > 1/2\},
    \end{align*}
    and the index $K_i'$ as 
    \begin{align*}
        K_i' := \sup \mathcal{S}'_i.
    \end{align*}
    Note that $\mathcal{S}'_i$ is never empty since $i\in \mathcal{S}'_i$ always holds. Trivially, if $\mathcal{S}'_i = \{i\}$, then $K_i' - i = 0$. Suppose that $|\mathcal{S}'_i|\geq 2$. By definition, we have $x_{K_i'} > x_i/2$. By assumption, $x_{K_i'}/x_j \leq C$ for all $i \leq j \leq K_i'$. It follows that
    \begin{align*}
        x_j \geq \frac{x_{K_i'}}{C} > \frac{x_i}{2C}, \quad \text{for } i \leq j \leq K_i'.
    \end{align*}
    Therefore, we have
    \begin{align*}
        C\geq \sum_{j=i}^{2i} \frac{x_j}{x_i} \geq \sum_{j=i}^{K_i'} \frac{x_j}{x_i} > \sum_{j=i}^{K_i'} \frac{1}{2C} = \frac{1}{2C}(K_i' - i + 1),
    \end{align*}
    implying that $(K_i' - i)$ is uniformly bounded above by $2C^2 -1$. Note that $2C^2 - 1 > 0$ because
    \begin{align*}
        C \geq \sum_{j = i}^{2i} \frac{x_j}{x_i} = 1 + \sum_{j =i+1}^{2i} \frac{x_j}{x_i} > 1.
    \end{align*}
    We define $\delta' := \lceil 2C^2 \rceil$. By definition, $x_{n + \delta'}/x_n \leq 1/2$ for all $n\in \N$. Now, for $m \geq n$, we have $m-n = k\delta' + s$ for some integers $k\geq 0$ and $0\leq  s \leq \delta' - 1$. Using the same arguments as in the first case, we have
    \begin{align*}
        \frac{x_m}{x_n} \leq \frac{x_{n+s}}{x_n} \cdot \left(2^{-1/\delta'} \right)^{m-n}2^{\delta'}.
    \end{align*}
    Different from the first case, it is not guaranteed that $x_{n+s}/x_n \leq C$ since $n+s$ could be greater than $2n$. Consider the decomposition $n+s = qn + l$ for some integers $q\geq 0$ and $0\leq l \leq n-1$. We have
    \begin{align*}
        \frac{x_{n+s}}{x_n} = \frac{x_{qn + l}}{x_n} = \frac{x_{qn + l}}{x_{qn}}\cdot \frac{x_{qn}}{x_n} \leq C \prod_{j=1}^{q-1} \frac{x_{(j+1)n}}{x_{jn}} \leq C^q.
    \end{align*}
    Substituting $q = (n+s - l)/n$, we get
    \begin{align*}
        C^q = C^{\frac{n+s-l}{n}} = C \cdot C^{s/n} \cdot (1/C)^{l/n} \leq C \cdot C^{s/n},
    \end{align*}
    where the inequality follows from the fact that $C\geq 1$. Since $s \leq \delta' -1$ and $n\geq 1$, we have
    \begin{align*}
        C^q \leq C\cdot C^{s/n} \leq C\cdot C^{\delta'-1} = C^{\delta'}.
    \end{align*}
    Putting everything together, we have shown that $x_m/x_n \leq \alpha r^{m-n}$, where $\alpha = (2C)^{\delta'}$ and $r =2^{-1/\delta'}$. This completes the proof of the proposition.
\end{proof}

%%%%%%%%%%%%%%%% PROOFS of THEOREM 2 %%%%%%%%%%%%%%%%
\section{Proof of Theorem \ref{thm: suff cond}}\label{sec: suff thm proof}
We begin by assuming that the hypothesis of Theorem \ref{thm: suff cond} is satisfied. As done in the previous section, for ease of presentation, we assume that $(b_n)$ is positive and note that the subsequent proofs do not change if $b_n$ is replaced with $|b_n|$. We now unpack some immediate consequences of the hypothesis. First, note that \ref{cond: dynamic ratio bounds} forces $(a_n)$ and $(|\lambda_n|)$ to decay at the same rate. Thus, combining with \ref{cond: a_n and lambda_n exponential decay}, there exists $\rho \in (0,1)$ such that $(a_n)$ and $(|\lambda_n|)$ both $(1, \rho)$-exponentially decay. Next, due to \ref{cond: b_n exponential decay} and \ref{cond: a_n/b_n exponential decay}, there exist $\alpha,\beta>0$ and $\mu,\nu\in (0,1)$ such that 
\begin{align}
    \frac{a_m/b_m}{a_n/b_n} \leq \alpha \mu^{m-n} \quad \text{and} \quad \frac{b_m}{b_n}\leq \beta \nu^{m-n}, \quad \text{for all } m\geq n.
    \label{defn: nu and mu}
\end{align}
Let $\omega := (\omega_n)$ and $\gamma^{-1} := (1/\gamma_n)$, where $(\gamma_n)$, $(\omega_n)$ are the sequences satisfying \ref{cond: dynamic ratio bounds}. Note that $\omega$, $\gamma^{-1}\in \ell^\infty$ by hypothesis. Since $|\lambda_n|/a_n \leq \|\omega\|_{\ell^\infty}$ and $a_n/|\lambda_n| \leq \|\gamma^{-1}\|_{\ell^\infty}$ for all $n\in \N$ by \ref{cond: dynamic ratio bounds}, it follows that for all $m \geq n$,
\begin{align}
    \frac{|\lambda_m|/b_m}{|\lambda_n|/b_n} \leq \alpha' \mu^{m - n}, \quad \text{where } \alpha' := \alpha \|\gamma^{-1}\|_{\ell^\infty} \|\omega\|_{\ell^\infty};
    \label{defn: alpha prime}
\end{align}
that is, $(|\lambda_n|/b_n)$ is $(\alpha', \mu)$-exponentially decaying.

Our first task is to prove that $k(\lambda)\in X^*$, which we accomplish by showing that $k(\lambda) \in \ell^1$ since $\ell^1 \subseteq X^*$ holds for all $X$. To this end, consider the following proposition.

\begin{proposition}
    It holds that $\pi(\lambda,a) \in \ell^\infty$ and $\pi(a, \lambda) \in \ell^\infty$.
    \label{Prop: pi(lambda, a) is uniformly bounded}
\end{proposition}

\noindent  We establish below that $k(\lambda) \in \ell^1$ is an immediate consequence of Proposition \ref{Prop: pi(lambda, a) is uniformly bounded}. We prove Proposition \ref{Prop: pi(lambda, a) is uniformly bounded} in Subsection \S \ref{subsec: proof that pi(lambda, a) is uniformly bounded}.

\begin{corollary}
    It holds that $k(\lambda) \in \ell^1$.
\end{corollary}
\begin{proof}[Proof.]
    By definition \eqref{defn: feedback k}, we have for all $n\in \N$,
    \begin{align*}
        |k_n(\lambda)| = \frac{a_n + |\lambda_n|}{b_n}|\pi_n(\lambda,a)|
    \end{align*}
    Hence,
    \begin{align*}
        \|k(\lambda)\|_{\ell^1} = \sum_{n = 1}^\infty |k_n(\lambda)| 
        &\leq \|\pi(\lambda,a)\|_{\ell^\infty} \sum_{n=1}^\infty \frac{a_n + |\lambda_n|}{b_n} \\
        &\leq \left( \|\omega\|_{\ell^\infty} + 1 \right)\|\pi(\lambda,a)\|_{\ell^\infty}  \sum_{n=1}^\infty \frac{a_n}{b_n}\\
        &\leq \frac{\alpha a_1}{b_1} \left( \|\omega\|_{\ell^\infty} + 1\right)\|\pi(\lambda,a)\|_{\ell^\infty}\sum_{n=1}^\infty \mu^{n-1}\\ &= \frac{\alpha a_1 \left(\|\omega\|_{\ell^\infty} + 1\right)\|\pi(\lambda,a)\|_{\ell^\infty}}{b_1(1-\mu)},
    \end{align*}
    where the first inequality follows from Proposition \ref{Prop: pi(lambda, a) is uniformly bounded}, the second inequality from \ref{cond: dynamic ratio bounds}, and the last inequality from \eqref{defn: nu and mu}. This completes the proof.
\end{proof}

Next, we show that $T_{k(\lambda)}$ is diagonalizable, which will require more effort. In accordance with Theorem \ref{thm: suff cond}, for the remainder of the section, we only consider the cases $X=\ell^p$ for $1\leq p \leq \infty$ and $X=c_0$. To show that $T_{k(\lambda)}$ is diagonalizable, we prove that $\tilde T$ is diagonalizable, where $\tilde T$ is the transformation of $T_{k(\lambda)}$ given by \eqref{defn: T tilde}. Recall the definition of $\tilde P$ and $\tilde Q$ in \eqref{defn: P and Q}, and recall that the columns of $\tilde P$ and rows of $\tilde Q$ represent the right and left eigenvectors of $\tilde T$, respectively. Setting the scaling factors to be $\sigma_n = \pi_n(a, \lambda)$ and $\tau_n = 1$ for $n\in \N$, we have
\begin{align}
\begin{split}
    \tilde P &= \left[\tilde P_{ij} := \frac{(a_j - \lambda_j)\pi_j(a, \lambda)}{a_i - \lambda_j} \right]_{1 \leq i,j < \infty}, \quad 
    \tilde Q = \left[\tilde Q_{ij} := \frac{(\lambda_j - a_j)\pi_j(\lambda,a)}{\lambda_i - a_j} \right]_{1 \leq i,j < \infty} .
\end{split}
\label{defn: new P and Q}
\end{align}
Note that $\tilde P$ and $\tilde Q$, as defined above, are symmetric in the sense that $\tilde Q$ is obtained from $\tilde P$ by swapping the roles of $a$ and $\lambda$.

Showing that $\tilde P$ and $\tilde Q$ of \eqref{defn: new P and Q} diagonalize $\tilde T$ takes two steps: first, we show that $\tilde P$ and $\tilde Q$ are well-defined, bounded operators from $Y$ to $Y$; second, we show that $\tilde P$ and $\tilde Q$ are inverses of each other. Beginning with the first step, we introduce the following constants:
\begin{align}
    \kappa :=\frac{(1+\|\gamma^{-1}\|_{\ell^\infty})\beta}{1-\nu} + \frac{(1 + \|\omega\|_{\ell^\infty})\alpha}{1-\mu}, \quad    \eta &:= \frac{(1 + \|\omega\|_{\ell^\infty})\beta}{1-\nu} + \frac{(1 + \|\gamma^{-1}\|_{\ell^\infty})\alpha'}{1-\mu},
    \label{defn: eta and kappa}
\end{align}
where $\|\gamma^{-1}\|_{\ell^\infty}$, $\|\omega\|_{\ell^\infty}$ come from \ref{cond: dynamic ratio bounds}, $\alpha, \beta>0$ and $\mu, \nu\in (0,1)$ satisfy \eqref{defn: nu and mu}, and $\alpha'$ is defined as in \eqref{defn: alpha prime}. The following proposition plays an important role in achieving the first step.

% option 
\begin{proposition}
The following hold:
\begin{enumerate}
    \item For all $i\in \N$,
        \begin{align}
            \sum_{j=1}^\infty \frac{(a_j + |\lambda_j|)/b_j}{(a_i + |\lambda_j|)/b_i} < \kappa \quad \text{and} \quad 
            \sum_{j=1}^\infty \frac{(|\lambda_j| + a_j)/b_j}{(|\lambda_i| + a_j)/b_i} < \eta. \label{eqn: kappa and eta bound (j sum)}
        \end{align}

    \item For all $j \in \N$,
        \begin{align}
            \sum_{i=1}^\infty \frac{(a_j + |\lambda_j|)/b_j}{(a_i + |\lambda_j|)/b_i} < \kappa \quad \text{and} \quad
            \sum_{i=1}^\infty \frac{(|\lambda_j| + a_j)/b_j}{(|\lambda_i| + a_j)/b_i} < \eta.
            \label{eqn: kappa and eta bound (i sum)}
        \end{align}
\end{enumerate}
\label{prop: kappa and eta bounds}
\end{proposition}
\noindent Equipped with Propositions \ref{prop: kappa and eta bounds} and \ref{Prop: pi(lambda, a) is uniformly bounded}, we are able to establish that $\tilde P$ and $\tilde Q$ are well-defined, bounded operators from $Y$ to $Y$.

\begin{proposition}
    The operators $\tilde P$ and $\tilde Q$, as defined in \eqref{defn: new P and Q}, map $Y$ to $Y$ and, moreover, satisfy
    \begin{align*}
        \| \tilde P\|_{\mathcal{B}(Y)} \leq \kappa \|\pi(a,\lambda)\|_{\ell^\infty} \quad \text{and} \quad \| \tilde Q \|_{\mathcal{B}(Y)} \leq \eta \|\pi(\lambda,a)\|_{\ell^\infty},
    \end{align*}
    where $\kappa$ and $\eta$ are given by \eqref{defn: eta and kappa}.
    \label{prop: P and Q are bounded}
\end{proposition}

% \begin{proposition}
%     The operators $\tilde P$ and $\tilde Q$, as defined in \eqref{defn: new P and Q}, map $Y$ to $Y$ and, moreover, satisfy $\| \tilde P\|_{\mathcal{B}(Y)} \leq \kappa \|\pi(a,\lambda)\|_{\ell^\infty}$ and $\| \tilde Q \|_{\mathcal{B}(Y)} \leq \eta \|\pi(\lambda,a)\|_{\ell^\infty}$, where $\kappa$ and $\eta$ are given by \eqref{defn: eta and kappa}.
%     \label{prop: P and Q are bounded}
% \end{proposition}

With the first step established, all that is left to do is show that $\tilde P$ and $\tilde Q$ are inverses of each other, which is accomplished by the following proposition.
\begin{proposition}
    The operators $\tilde P$ and $\tilde Q$, as defined in \eqref{defn: new P and Q}, satisfy $\tilde P \tilde Q = \tilde Q \tilde P = I$.
    \label{prop: P and Q are inverses}
\end{proposition}
\noindent Collectively, Propositions \ref{prop: P and Q are bounded} and \ref{prop: P and Q are inverses} establish that $\tilde T$ is diagonalizable by $\tilde P$ and $\tilde Q$. We prove propositions \ref{prop: kappa and eta bounds}, \ref{prop: P and Q are bounded}, and \ref{prop: P and Q are inverses} in Subsections \S \ref{subsec: proof of kappa and eta bounds}, \S \ref{subsec: proof that P, Q are bounded}, and \S \ref{subsec: proof that P, Q are inverses}, respectively.

\subsection{Proof of Proposition \ref{Prop: pi(lambda, a) is uniformly bounded}}\label{subsec: proof that pi(lambda, a) is uniformly bounded}
\begin{proof}[Proof.]
    We show that $(\log|\pi_n(\lambda,a)|) \in \ell^\infty$ and $(\log|\pi_n(a, \lambda)|) \in \ell^\infty$. From definition \eqref{defn: pi(lambda, a)}, 
    \begin{align*}
        \log|\pi_n(\lambda,a)| = \sum_{\substack{m=1\\m\neq n}}^\infty \log \left| \frac{1 - \lambda_m/a_n}{1 - a_m/a_n} \right| = \sum_{m=1}^{n-1} \log \left| \frac{1 - \lambda_m/a_n}{1 - a_m/a_n} \right| + \sum_{m=n+1}^\infty \log \left|\frac{1 -\lambda_m/a_n}{1 - a_m/a_n} \right|.
    \end{align*}
    We provide upper bounds for the sums on the right hand side through the two cases below.

    \medskip
    \noindent \textit{Case 1: $m < n$.} Recall that $\rho \in (0,1)$ satisfies $a_m/a_n \leq \rho$ for $m\geq n$. Hence, for this case we have $a_m/a_n \geq \rho^{m - n} > 1$, which implies that
    \begin{align*}
        \left|\frac{1 -\lambda_m/a_n}{1 - a_m/a_n} \right| = \frac{1 + |\lambda_m|/a_n}{a_m/a_n - 1} &= \frac{a_n/a_m + |\lambda_m|/a_m}{1 - a_n/a_m} \\
        &= 1 + \frac{2a_n/a_m + |\lambda_m|/a_m -1}{1 - a_n/a_m} \\
        &\leq 1 + \frac{2\rho^{n-m} + \omega_m -1}{1 - \rho^{n-m}} \\&< 1 + \frac{2\rho^{n-m} + \omega_m -1}{1 - \rho}.
    \end{align*}
    Note that the last inequality holds because $(\omega_n - 1)$ is assumed to be a positive sequence by \ref{cond: dynamic ratio bounds}. It then follows that
    \begin{align*}
        \sum_{m=1}^{n-1} \log \left| \frac{1 - \lambda_m/a_n}{1 - a_m/a_n} \right| &\leq \sum_{m=1}^{n-1} \log \left( 1 + \frac{2\rho^{n-m} + \omega_m -1}{1 - \rho} \right) \\&\leq \frac{1}{1 - \rho} \sum_{m = 1}^{n-1} (2\rho^{n-m} + \omega_m - 1) \\ &< \frac{2 \rho}{(1 - \rho)^2} + \frac{ \|(\omega_n-1)\|_{\ell^1} }{1-\rho},
    \end{align*}
    where the last inequality follows from \ref{cond: dynamic ratio bounds}.

    \medskip
    \noindent \textit{Case 2: $m > n$.} Here, we have $a_m/a_n \leq \rho ^{m-n} < 1$. Hence, 
    \begin{align*}
        \left|\frac{1 -\lambda_m/a_n}{1 - a_m/a_n} \right| = \frac{1 + |\lambda_m|/a_n}{1 - a_m/a_n} &\leq \frac{1 + a_m\|\omega\|_{\ell^\infty}/a_n}{1 - a_m/a_n} \\
        &= 1 + \frac{(\|\omega\|_{\ell^\infty} + 1) a_m/a_n}{1 - a_m/a_n} \\
        &\leq 1 + \frac{(\|\omega\|_{\ell^\infty} + 1) \rho^{m-n}}{1 - \rho^{m-n}} \\
        &\leq  1 + \frac{(\|\omega\|_{\ell^\infty} + 1) \rho^{m-n}}{1 - \rho}.
    \end{align*}
    Therefore, we have
    \begin{align*}
        \sum_{m=n+1}^\infty \log \left|\frac{1 -\lambda_m/a_n}{1 - a_m/a_n} \right| &\leq \sum_{m=n+1}^\infty \log \left(  1 + \frac{(\|\omega\|_{\ell^\infty} + 1) \rho^{m-n}}{1 - \rho} \right) \\
        &\leq \frac{\|\omega\|_{\ell^\infty} + 1}{1 - \rho} \sum_{m=n+1}^\infty \rho^{m-n} = \frac{(\|\omega\|_{\ell^\infty} + 1)\rho}{(1 - \rho)^2}.
    \end{align*}
    Combining the above arguments, we conclude that for all $n\in \N$,
    \begin{align*}
        \log |\pi_n(\lambda,a)| &= \sum_{m=1}^{n-1} \log \left| \frac{1 - \lambda_m/a_n}{1 - a_m/a_n} \right| + \sum_{m=n+1}^\infty \log \left|\frac{1 -\lambda_m/a_n}{1 - a_m/a_n} \right| \\&< \frac{(\|\omega\|_{\ell^\infty} + 3)\rho}{(1- \rho)^2} + \frac{\|(\omega_n - 1)\|_{\ell^1}}{1-\rho}.
    \end{align*}
    Since $(|\lambda_n|)$ and $(a_n)$ both $(1,\rho)$-exponentially decay, and since $a_n/|\lambda_n| \leq 1/\gamma_n \leq \|\gamma^{-1}\|_{\ell^\infty}$ due to \ref{cond: dynamic ratio bounds}, we are able to use similar arguments as above to show that for all $n\in \N$,
    \begin{align*}
        \log|\pi_n(a, \lambda)| <  \frac{(\|\gamma^{-1}\|_{\ell^\infty} + 3)\rho}{(1- \rho)^2} + \frac{\|(1/\gamma_n - 1)\|_{\ell^1}}{1-\rho}.
    \end{align*}
    This completes the proof.
\end{proof}

%%%%%%%%%%%%%%%%%%%%%%%%%%%%%%%%%%%%%%%%%%%%%%%%
\subsection{Proof of Proposition \ref{prop: kappa and eta bounds}}\label{subsec: proof of kappa and eta bounds}
We establish below the first inequality of both \eqref{eqn: kappa and eta bound (j sum)} and \eqref{eqn: kappa and eta bound (i sum)}, i.e., the $\kappa$ bounds. Then, by symmetry, the second inequality of both \eqref{eqn: kappa and eta bound (j sum)} and \eqref{eqn: kappa and eta bound (i sum)} can be established using the same arguments but with $a$ and $\lambda$ swapping roles.
\begin{proof}[Proof.]
We begin by splitting the sum in the first inequality of \eqref{eqn: kappa and eta bound (j sum)} into two parts as follows:
\begin{align*}
    \sum_{j=1}^\infty \frac{(a_j + |\lambda_j|)/b_j}{(a_i + |\lambda_j|)/b_i} = \sum_{j=1}^{i - 1} \frac{(a_j + |\lambda_j|)/b_j}{(a_i + |\lambda_j|)/b_i} + \sum_{j=i}^\infty \frac{(a_j + |\lambda_j|)/b_j}{(a_i + |\lambda_j|)/b_i}.
\end{align*}
Looking at the first sum, we have the bound
\begin{align*}
    \sum_{j=1}^{i-1} \frac{(a_j + |\lambda_j|)/b_j}{(a_i + |\lambda_j|)/b_i} &< \sum_{j=1}^{i-1} \frac{(a_j + |\lambda_j|)/b_j}{|\lambda_j|/b_i} \\
    &\leq  (1 + \|\gamma^{-1}\|_{\ell^\infty}) \sum_{j=1}^{i-1} \frac{b_i}{b_j} \\
    &\leq (1 + \|\gamma^{-1}\|_{\ell^\infty})\beta \sum_{j=1}^{i-1}\nu^{i-j} < \frac{(1 + \|\gamma^{-1}\|_{\ell^\infty})\beta}{1-\nu},
\end{align*}
where the second inequality follows from \ref{cond: dynamic ratio bounds} and the third inequality follows from \eqref{defn: nu and mu}. Applying \ref{cond: dynamic ratio bounds} and \eqref{defn: nu and mu} in the same manner as above, the second sum is bounded by
\begin{align*}
    \sum_{j=i}^\infty \frac{(a_j + |\lambda_j|)/b_j}{(a_i + |\lambda_j|)/b_i} &< \sum_{j=i}^\infty \frac{(a_j + |\lambda_j|)/b_j}{a_i/b_i} \\
    &\leq (1 + \|\omega\|_{\ell^\infty}) \sum_{j=i}^\infty \frac{a_j/b_j}{a_i/b_i} \\
    &\leq (1 + \|\omega\|_{\ell^\infty})\alpha \sum_{j=i}^\infty \mu^{j-i} = \frac{(1 + \|\omega\|_{\ell^\infty})\alpha}{1-\mu}.
\end{align*}
In conclusion, we have shown that for all $i\in \N$,
\begin{align*}
    \sum_{j=1}^\infty \frac{(a_j + |\lambda_j|)/b_j}{(a_i + |\lambda_j|)/b_i} &= \sum_{j=1}^{i-1} \frac{(a_j + |\lambda_j|)/b_j}{(a_i + |\lambda_j|)/b_i} + \sum_{j=i}^\infty \frac{(a_j + |\lambda_j|)/b_j}{(a_i + |\lambda_j|)/b_i}  \\&< \frac{(1 + \|\gamma^{-1}\|_{\ell^\infty})\beta}{1-\nu} + \frac{(1 + \|\omega\|_{\ell^\infty})\alpha}{1-\mu} = \kappa,
\end{align*}
where the last equality follows from definition \eqref{defn: eta and kappa}. This establishes the $\kappa$ bound for \eqref{eqn: kappa and eta bound (j sum)}; we now prove the $\kappa$ bound for \eqref{eqn: kappa and eta bound (i sum)}. Using similar arguments as above, we have
\begin{align*}
    \sum_{i=1}^{j-1} \frac{(a_j + |\lambda_j|)/b_j}{(a_i + |\lambda_j|)/b_i} &< \sum_{i=1}^{j-1} \frac{(a_j + |\lambda_j|)/b_j}{a_i/b_i} 
    \leq (1 + \|\omega\|_{\ell^\infty})\alpha \sum_{i=1}^{j-1} \mu^{j-i} < \frac{(1 + \|\omega\|_{\ell^\infty})\alpha}{1-\mu},
\end{align*}
and
\begin{align*}
    \sum_{i=j}^\infty \frac{(a_j + |\lambda_j|)/b_j}{(a_i + |\lambda_j|)/b_i} &< \sum_{i=j}^\infty \frac{(a_j + |\lambda_j|)/b_j}{|\lambda_j|/b_i} \leq (1 + \|\gamma^{-1}\|_{\ell^\infty})\beta \sum_{i=j}^\infty \nu^{i-j} = \frac{(1 + \|\gamma^{-1}\|_{\ell^\infty})\beta}{1-\nu}.
\end{align*}
Hence, for all $j\in \N$,
\begin{align*}
    \sum_{i=1}^\infty \frac{(a_j + |\lambda_j|)/b_j}{(a_i + |\lambda_j|)/b_i} &= \sum_{i=1}^{j-1} \frac{(a_j + |\lambda_j|)/b_j}{(a_i + |\lambda_j|)/b_i} + \sum_{i=j}^\infty \frac{(a_j + |\lambda_j|)/b_j}{(a_i + |\lambda_j|)/b_i} \\& < \frac{(1 + \|\omega\|_{\ell^\infty})\alpha}{1-\mu} + \frac{(1 + \|\gamma^{-1}\|_{\ell^\infty})\beta}{1-\nu} = \kappa,
\end{align*}
which establishes the first inequality for \eqref{eqn: kappa and eta bound (i sum)}. This completes the proof.
\end{proof}

%%%%%%%%%%%%%%%%%%%%%%%%%%%%%%%%%%%%%%%%%%%%%%%%%
\subsection{Proof of Proposition \ref{prop: P and Q are bounded}}\label{subsec: proof that P, Q are bounded}
\begin{proof}[Proof.]
    We first establish that $\tilde P$ maps $Y$ to $Y$ and satisfies $\|\tilde P\|_{\mathcal{B}(Y)} \leq \kappa \|\pi(a, \lambda)\|_{\ell^\infty}$. Let $y \in Y$ and define
    \begin{align*}
        y' := \tilde Py = \left(\sum_{j = 1}^\infty \frac{(a_j - \lambda_j)\pi_j(a,\lambda)y_j}{a_i - \lambda_j} \right)_{i \in \N}.
    \end{align*}
    We show that
    \begin{align}
        \|y'\|_{Y} \leq \kappa \|\pi(a, \lambda)\|_{\ell^\infty} \|y\|_Y
        \label{eqn: y prime norm}
    \end{align}
    for the two following cases of $X$: (1) $X=\ell^\infty$ or $X=c_0$, and (2) $X=\ell^p$ for $1\leq p < \infty$.

    \medskip
    \noindent \textit{Case 1: $X=\ell^\infty$ or $c_0$}. For every $i\in \N$, we have
    \begin{align*}
        |b_i y'_i| = \left|\sum_{j = 1}^\infty \frac{(a_j - \lambda_j)\pi_j(a,\lambda)y_jb_i}{a_i - \lambda_j} \right| 
        &\leq \sum_{j = 1}^\infty \frac{(a_j + |\lambda_j|)|\pi_j(a,\lambda)|}{a_i + |\lambda_j|}|y_j b_i| \\
        &\leq \|\pi(a, \lambda)\|_{\ell^\infty} \sum_{j = 1}^\infty \frac{(a_j + |\lambda_j|)/b_j}{(a_i + |\lambda_j|)/b_i}|y_jb_j| \\
        &\leq \|\pi(a, \lambda)\|_{\ell^\infty} \|y\|_{Y} \sum_{j = 1}^\infty \frac{(a_j + |\lambda_j|)/b_j}{(a_i + |\lambda_j|)/b_i} \\
        &< \kappa \|\pi(a, \lambda)\|_{\ell^\infty} \|y\|_{Y},
    \end{align*}
    where the last inequality follows from \eqref{eqn: kappa and eta bound (j sum)}. Hence, we have shown that
    \begin{align*}
        \|y'\|_Y = \sup_{i\in \N} |b_i y'_i| < \kappa \|\pi(a, \lambda)\|_{\ell^\infty} \|y\|_Y.
    \end{align*}
    For the case where $X=c_0$, we must further show that $(b_n y'_n) \in c_0$ to conclude that $y'\in Y$. For this case of $X$, we have that $(b_n y_n) \in c_0$. Hence, for all $\epsilon > 0$, there exists a positive integer $N$ such that 
    \begin{align}
        |b_i y_i| < \frac{\epsilon}{2\kappa \|\pi(a, \lambda)\|_{\ell^\infty}}, \quad \text{for all }i\geq N.
        \label{eqn: b_n y_n epsilon bound}
    \end{align}
    Moreover, note that for any fixed $j\in \N$, we have
    \begin{align*}
        \lim_{i\to \infty} \frac{(a_j + |\lambda_j|)/b_j}{(a_i + |\lambda_j|)/b_i} = \frac{a_j + |\lambda_j|}{b_j} \lim_{i\to \infty} \frac{b_i}{a_i + |\lambda_j|} = 0,
    \end{align*}
    since $(a_n)$, $(b_n)$ are null sequences by hypothesis. It follows that there exists an integer $N' \geq N$, where $N$ is the integer satisfying \eqref{eqn: b_n y_n epsilon bound}, such that
    \begin{align}
        0 < \frac{(a_j + |\lambda_j|)/b_j}{(a_i + |\lambda_j|)/b_i} < \frac{\epsilon}{2N\|\pi(a, \lambda)\|_{\ell^\infty}\|y\|_Y}, \quad \text{for all } j \leq N \text{ and for all }i\geq N'.
        \label{eqn: sum term epsilon bound}
    \end{align}
    Therefore, for all $i \geq N'$, we have
    \begin{align*}
        |b_i y_i'| &\leq \|\pi(a,\lambda)\|_{\ell^\infty}\left[ \sum_{j = 1}^N \frac{(a_j + |\lambda_j|)/b_j}{(a_i + |\lambda_j|)/b_i} |b_j y_j| + \sum_{j = N+1}^\infty \frac{(a_j + |\lambda_j|)/b_j}{(a_i + |\lambda_j|)/b_i} |b_j y_j| \right]\\
        &< \frac{\epsilon}{2N\|y\|_Y} \sum_{j=1}^N |b_j y_j| + \frac{\epsilon}{2\kappa} \sum_{j = N+1}^\infty \frac{(a_j + |\lambda_j|)/b_j}{(a_i + |\lambda_j|)/b_i}
        < \frac{\epsilon}{2} + \frac{\epsilon}{2} = \epsilon,
    \end{align*}
    where the second inequality follows from \eqref{eqn: b_n y_n epsilon bound} and \eqref{eqn: sum term epsilon bound} and the last inequality from \eqref{eqn: kappa and eta bound (j sum)}. This completes the proof that $(b_ny'_n)\in c_0$ and, thus, completes the proof of Case 1.
    \medskip

    \noindent \textit{Case 2: $X = \ell^p$ for $1\leq p < \infty$}. Let $1 < q \leq \infty$ be the H\"older's conjugate to $p$, i.e., $q$ and $p$ satisfy $1/p + 1/q = 1$. For notational convenience, we define for each $i\in \N$,
    \begin{align*}
        \phi_i := \left(\left[ \frac{(a_j + |\lambda_j|)/b_j}{(a_i + |\lambda_j|)/b_i}\right]^\frac{1}{q} \right)_{j \in \N}, \quad \psi_i :=  \left(\left[ \frac{(a_j + |\lambda_j|)/b_j}{(a_i + |\lambda_j|)/b_i}\right]^\frac{1}{p} |b_jy_j| \right)_{j \in \N}.
    \end{align*} 
    Then, for each $i\in \N$, we have
    \begin{align*}
        |b_i y_i'| = \left|\sum_{j = 1}^\infty \frac{(a_j - \lambda_j)\pi_j(a,\lambda)y_jb_i}{a_i - \lambda_j} \right|
        &\leq \sum_{j = 1}^\infty \frac{(a_j + |\lambda_j|)|\pi_j(a, \lambda)|/b_j}{(a_i + |\lambda_j|)/b_i}|b_jy_j|\\
        &\leq \|\pi(a, \lambda)\|_{\ell^\infty} \sum_{j=1}^\infty \phi_i\psi_i = \|\pi(a, \lambda)\|_{\ell^\infty} \|\phi_i \psi_i\|_{\ell^1}.
    \end{align*}
    Note that $\|\phi_i \|_{\ell^q} < \kappa^{1/q}$ and $\|\psi_i \|_{\ell^p} < \kappa^{1/p} \|y\|_Y$ due to Proposition \ref{prop: kappa and eta bounds}. It follows that 
    \begin{align}
        |b_i y_i'| \leq \|\pi(a, \lambda)\|_{\ell^\infty} \|\phi_i \psi_i \|_{\ell^1} \leq  \|\pi(a, \lambda)\|_{\ell^\infty}\|\phi_i\|_{\ell^q} \|\psi_i \|_{\ell^p} < \kappa^\frac{1}{q} \|\pi(a, \lambda)\|_{\ell^\infty} \|\psi_i \|_{\ell^p},
        \label{eqn: Holder's bound}
    \end{align}
    where the second inequality follows from H\"older's inequality. Finally, it follows that
    \begin{align*}
        \|y'\|_Y^p = \sum_{i=1}^\infty |b_i y'_i|^p &< \kappa^{\frac{p}{q}} \|\pi(a,\lambda)\|_{\ell^\infty}^p \sum_{i=1}^\infty \|\psi_i \|_{\ell^p}^p \\
        &= \kappa^{\frac{p}{q}} \|\pi(a,\lambda)\|_{\ell^\infty}^p \sum_{i=1}^\infty \sum_{j=1}^\infty \frac{(a_j + |\lambda_j|)/b_j}{(a_i + |\lambda_j|)/b_i} |b_j y_j|^p \\
        &= \kappa^{\frac{p}{q}} \|\pi(a,\lambda)\|_{\ell^\infty}^p \sum_{j=1}^\infty |b_j y_j|^p \sum_{i=1}^\infty \frac{(a_j + |\lambda_j|)/b_j}{(a_i + |\lambda_j|)/b_i} \\
        &< \kappa^{\frac{p}{q} + 1} \|\pi(a,\lambda)\|_{\ell^\infty}^p \sum_{j=1}^\infty |b_j y_j|^p = \kappa^{\frac{p}{q} + 1} \|\pi(a,\lambda)\|_{\ell^\infty}^p \|y\|_Y^p,
    \end{align*}
    where the first inequality follows from \eqref{eqn: Holder's bound} and the last inequality from \eqref{eqn: kappa and eta bound (j sum)}. In conclusion, we have shown that
    \begin{align*}
         \|y'\|_Y < \kappa^{\frac{1}{q} + \frac{1}{p}} \|\pi(a,\lambda)\|_{\ell^\infty} \|y\|_Y = \kappa \|\pi(a,\lambda)\|_{\ell^\infty} \|y\|_Y,
    \end{align*}
    which establishes \eqref{eqn: y prime norm} for Case 2. 
    
    \medskip
    We now turn our attention to $\tilde Q$. For $y\in Y$, we re-define $y' := \tilde Q y$. To establish that $y'\in Y$ and that $\|y'\|_Y \leq \eta \|\pi(\lambda,a)\|_{\ell^\infty}\|y\|_Y$, due to the symmetry between $\tilde Q$ and $\tilde P$, we can use the same arguments as above but with $\|\pi(\lambda,a)\|_{\ell^\infty}$ replacing $\|\pi(a, \lambda)\|_{\ell^\infty}$ and the $\eta$ bounds of \eqref{eqn: kappa and eta bound (j sum)} and \eqref{eqn: kappa and eta bound (i sum)} replacing the $\kappa$ bounds. This completes the proof.
\end{proof}

%%%%%%%%%%%%%%%%%%%%%%%%%%%%%%%%%%%%%%%%%%%%%%%%%
\subsection{Proof of Proposition \ref{prop: P and Q are inverses}}\label{subsec: proof that P, Q are inverses}
\begin{proof}[Proof.]
    Let $\tilde P_j \in Y$ be the $j$th column of $\tilde P$ and $\tilde Q_i^\top \in Y^*$ be the $i$th row of $\tilde Q$. We establish below that
    \begin{align}
        \tilde Q_i^\top \tilde P_j = \delta_{ij}.
        \label{eqn: desired result}
    \end{align}
    For $N\in \N$, let $\pi(\lambda,a; N)$ and $\pi(a, \lambda; N)$ be given by \eqref{defn: N approx of pi sequence}.
    Similar to what was done in \S \ref{subsec: proof of prop: scaling properties}, we introduce approximations $\tilde P(N)$ and $\tilde Q(N)$ of $\tilde P$ and $\tilde Q$, respectively, as follows:
    \begin{align*}
        \tilde P(N) &= \left[\tilde P_{ij}(N) := \frac{(a_j - \lambda_j)\pi_j(a, \lambda; N)}{a_i - \lambda_j} \right]_{1 \leq i,j < \infty}, \\
        \tilde Q(N) &= \left[\tilde Q_{ij}(N) := \frac{(\lambda_j - a_j)\pi_j(\lambda,a; N)}{\lambda_i - a_j} \right]_{1 \leq i,j < \infty}.
    \end{align*}
    Next, we define the $N\times N$ matrices
    \begin{align*}
        \tilde Q'(N) &:= \left[\tilde Q_{ij}(N)\right]_{1 \leq i,j \leq N}, \quad
        \tilde P'(N) := \left[\tilde P_{ij}(N)\right]_{1\leq i,j\leq N}.
    \end{align*}
    By construction, we have
    \begin{align}
        \tilde Q_i^\top(N) \tilde P_j(N) = \tilde Q_i '^\top(N) \tilde P_j'(N).
        \label{eqn: QP (N) related to QP prime (N)}
    \end{align}

    % % option 1
    % \noindent Next, we define the $N\times N$ matrices $A:=G$ and $B:=G^{-1}$, where $G$ is the Cauchy matrix given by \eqref{eqn: Cauchy matrix} and $G^{-1}$ is its inverse given by \eqref{eqn: Cauchy inverse}. Note that
    % \begin{align*}
    %     \tilde Q_i'^\top(N) =  \frac{1}{(a_i-\lambda_i)\pi_i(a, \lambda)}B_i^\top, \quad \text{and} \quad 
    %     \tilde P_j'(N)  = (a_j - \lambda_j)\pi_j(a,\lambda)A_j.
    % \end{align*}
    % It follows that
    % \begin{align*}
    %     \tilde Q_i '^\top(N) \tilde  P_j'(N) = \frac{(a_j - \lambda_j)\pi_j(a,\lambda)}{(a_i - \lambda_i)\pi_i(a, \lambda)} B_i^\top A_j = \delta_{ij}.
    % \end{align*}

    % option 2
    \noindent Next, recall the $N \times N$ Cauchy matrix $G$ given by \eqref{eqn: Cauchy matrix} and its inverse $G^{-1}$ given by \eqref{eqn: Cauchy inverse}. We note that
    \begin{align*}
        \tilde Q_i'^\top(N) =  \frac{1}{(a_i-\lambda_i)\pi_i(a, \lambda)} e_i^\top G^{-1} \quad \text{and} \quad 
        \tilde P_j'(N)  = (a_j - \lambda_j)\pi_j(a,\lambda)G e_j,
    \end{align*}
    where $e_i^\top G^{-1}$ is the $i$th row of $G^{-1}$ and $Ge_j$ is the $j$th column of $G$. Hence,
    \begin{align*}
        \tilde Q_i '^\top(N) \tilde  P_j'(N) = \frac{(a_j - \lambda_j)\pi_j(a,\lambda)}{(a_i - \lambda_i)\pi_i(a, \lambda)} e_i^\top G^{-1} Ge_j = \delta_{ij}.
    \end{align*}
    
    \noindent Applying this result to \eqref{eqn: QP (N) related to QP prime (N)}, we have thus shown that
    \begin{align}
        \tilde Q_i^\top(N) \tilde P_j(N) = \delta_{ij}.
        \label{eqn: so close to the end}
    \end{align}
    
    Our final step of the proof is to show that $\tilde Q_i^\top(N) \to \tilde Q_i^\top$ and $\tilde P_j(N) \to \tilde P_j$ as $N\to \infty$. We establish the strong convergence of $\tilde Q_i^\top(N) \to \tilde Q_i^\top$ below. Due to the symmetry between $\tilde P$ and $\tilde Q$, the strong convergence of $\tilde P_j(N) \to \tilde P_j$ can be shown using similar arguments.

    \medskip
    \noindent \textit{Strong convergence of $\tilde Q_i^\top(N)$.}
    We prove here that
    \begin{align}
        \lim_{N\to \infty} \|\tilde Q_i^\top - \tilde Q_i^\top(N)\|_{Y^*} = 0.
        \label{eqn: strong convergence of Q(N)}
    \end{align}
    Let $Q_i^\top := (\tilde Q_{ij}/b_j)_{j\in \N}$ and $Q_i^\top (N) := (Q_{ij}(N))_{j\in \N}$. In order to establish \eqref{eqn: strong convergence of Q(N)}, it suffices to show that
    \begin{align*}
        \lim_{N\to \infty} \|Q_i^\top - Q_i^\top(N) \|_{\ell^1} = 0.
    \end{align*}
    For each $i\in \N$, we have that $Q_i^\top \in \ell^1$, as shown below:
    \begin{align*}
        \| Q_i^\top \|_{\ell^1} = \sum_{j = 1}^\infty |\tilde Q_{ij}|/b_j 
        &= \sum_{j=1}^\infty \frac{(|\lambda_j| + a_j) |\pi_j(\lambda,a)| /b_j}{|\lambda_i| + a_j} \\
        &\leq \frac{\|\pi(\lambda,a)\|_{\ell^\infty}}{b_i} \sum_{j=1}^\infty \frac{(|\lambda_j| + a_j)/b_j}{(|\lambda_i| + a_j)/b_i} \\
        &< \frac{\eta \|\pi(\lambda,a)\|_{\ell^\infty}}{b_i},
    \end{align*}
    where the last inequality follows from \eqref{eqn: kappa and eta bound (j sum)}. Hence, for any $\epsilon > 0$, there exists a positive integer $N'$ such that
    \begin{align}
        \sum_{j = N' + 1}^\infty | Q_{ij}| \leq \frac{\epsilon}{2},
        \label{eqn: Q tilde is in ell 1}
    \end{align}
    Moreover, using similar arguments as those used to prove Lemma \ref{lem: convergence of Q(N)} in \S \ref{subsubsec: convegence of Q(N)}, there exists an integer $N_{\epsilon} \geq N'$, where $N'$ satisfies \eqref{eqn: Q tilde is in ell 1}, such that
    \begin{align}
        0 < 1- \frac{\pi_j(\lambda,a; N)}{ \pi_j(\lambda,a)} \leq \frac{\epsilon}{2\|Q_i^\top \|_{\ell^1}}, \quad \text{ for all } j\leq N' \text{ and for all } N \geq N_\epsilon.
        \label{eqn: that other one I need}
    \end{align}

    \noindent Then, using the same arguments that established item 1 of Lemma \ref{lem: convergence of Q(N)}, but for the particular case where $q=1$ and with \eqref{eqn: Q tilde is in ell 1} and \eqref{eqn: that other one I need} replacing \eqref{eqn: existence of N prime} and \eqref{eqn: existence of N epsilon}, respectively, it follows that for all $N\geq N_{\epsilon}$, 
    \begin{align*}
        \| Q_i^\top - Q_i^\top(N) \|_{\ell^1} < \epsilon.
    \end{align*}
     This completes the proof that $\lim_{N\to \infty} \| Q_i^\top - Q_i^\top(N) \|_{\ell^1} =0$, which thus establishes \eqref{eqn: strong convergence of Q(N)}. As noted earlier, due to the symmetry between $\tilde Q$ and $\tilde P$, we are able to show that
    \begin{align*}
        \lim_{N\to \infty} \|\tilde P_j - \tilde P_j(N)\|_Y = 0
    \end{align*}
    using similar arguments. Finally, applying the above convergence results to \eqref{eqn: so close to the end}, we have
    \begin{align*}
        \tilde Q_i^\top \tilde P_j = \lim_{N\to \infty} \tilde Q_i^\top(N) \tilde P_j(N) = \delta_{ij},
    \end{align*}
    which establishes the desired result \eqref{eqn: desired result} and completes the proof.
\end{proof}

%%%%%%%%%%%%%%%% CONCLUSION %%%%%%%%%%%%%%%%

\section{Conclusion}
In this paper, we have provided a partial solution to the following open problem: when is it possible to stabilize a discrete linear ensemble system via feedback control? In particular, we addressed pole placement and feedback diagonalizability, which together imply feedback stabilization, as shown in Corollary \ref{cor: closed-loop stability}. In Theorem \ref{thm: nec cond}, we have established that for pole placement and feedback diagonalizability to be feasible, it must hold that the following sequences uniformly exponentially decay: $(a_n)$, $(|\lambda_n|)$, $(|b_n|)$ and, under the additional assumption that $(a_n)$ does not decay faster than exponentially, $(a_n/|b_n|)$. Conversely, in Theorem \ref{thm: suff cond} we have proven that the uniform exponential decay of these sequences, along with a condition which, roughly speaking, says that $|\lambda_n|/a_n$ is uniformly bounded above and below and converges to one, guarantees that pole placement and feedback diagonalizability are feasible. The proofs of these two theorems are provided in Sections \S \ref{sec: proof of nec cond} and \S \ref{sec: suff thm proof}, respectively. Note that our proofs are constructive, i.e., we provide an explicit formula for the feedback law and compute the diagonalizing linear operators. We conclude by posing an open problem for future work. Is the uniform exponential decay of $(|b_n|)$ and $(a_n/|b_n|)$ also necessary for feedback stabilizability, not merely pole placement and feedback diagonalizability?

%%%%%%%%%%%%%%%% bibliography %%%%%%%%%%%%%%%%%
\printbibliography

\end{document}